\documentclass[11pt, a4paper]{amsart}
\usepackage{amssymb, array, amsmath, amscd, pdfpages, enumerate, amsthm}
\usepackage[T1]{fontenc}
\usepackage[colorlinks=true,allcolors=magenta]{hyperref}

\DeclareMathOperator*{\Ran}{Ran}
\DeclareMathOperator*{\Ker}{Ker}
\DeclareMathOperator*{\R}{Re}
\DeclareMathOperator*{\diag}{diag}

\newcommand{\inv}{^{-1}}
\newcommand{\<}{\langle}
\renewcommand{\>}{\rangle}
\newcommand{\dd}{\mathrm{d}}
\newcommand{\B}{\mathcal{B}}

\newcommand{\RR}{\mathbb{R}}
\newcommand{\C}{\mathbb{C}}

\newcommand{\T}{(T(t))_{t\ge 0}}

\newcommand{\ga}{\alpha}
\newcommand{\gl}{\lambda}
\newcommand{\iprod}[2]{\langle#1,#2\rangle}
\newcommand{\norm}[1]{\|#1\|}
\newcommand{\abs}[1]{|#1|}

\newtheorem{thm}{Theorem}[section]
\newtheorem{prp}[thm]{Proposition}
\newtheorem{lem}[thm]{Lemma}
\newtheorem{cor}[thm]{Corollary}
\theoremstyle{definition}
\newtheorem{rem}[thm]{Remark}
\newtheorem{ex}[thm]{Example}

\numberwithin{equation}{section}

\begin{document}

\title{Stability of abstract coupled systems}

\author[S.\ Nicaise]{Serge Nicaise}
\address[S.\ Nicaise]{Universit\'e Polytechnique Hauts-de-France, CERAMATHS/DEMAV, Valenciennes, France}
\email{serge.nicaise@univ-valenciennes.fr}

\author[L.~Paunonen]{Lassi Paunonen}
\address[L.~Paunonen]{Mathematics Research Centre, Tampere University, P.O.~ Box 692, 33101 Tampere, Finland}
 \email{lassi.paunonen@tuni.fi}

\author[D. Seifert]{David Seifert}
\address[D. Seifert]{School of Mathematics, Statistics and Physics, Newcastle University, Herschel Building, Newcastle upon Tyne, NE1 7RU, UK}
\email{david.seifert@ncl.ac.uk}

\begin{abstract}
We study stability of abstract differential equations  coupled by means of a general algebraic condition. Our approach is based on  techniques from operator theory and systems theory, and it allows us to study coupled systems by exploiting  properties of the  components, which are typically much simpler to analyse. As our main results we establish resolvent estimates and decay rates for abstract boundary-coupled systems. We illustrate the power of the general results by using them to obtain rates of energy decay in coupled systems of one-dimensional wave and heat equations, and in a multi-dimensional wave equation with an acoustic boundary condition.
\end{abstract}

\subjclass[2010]{
  47D06, % one-parameter semigroups linear evolution equations
47A10, % Operator theory: spectrum, resolvent
34G10, % ODEs -> ODEs on abstract spaces -> Linear differential equations in abstract spaces
35B35 %PDEs -> qualitative properties of solutions -> stability 
 %93B52 %Feedback control
 (35M30,  % PDEs -> Partial differential equations of mixed type and mixed-type systems of partial differential equations -> Mixed-type systems of PDEs
35L05, % PDE -> Wave equation
93D15)% Systems theory -> Stabilization of systems by feedback 
}
\keywords{
Boundary-coupled systems, 
operator semigroups,
stability, 
resolvent estimates,
boundary nodes,
energy decay rates,
wave-heat network,
 dynamic boundary condition}

 \thanks{The authors gratefully acknowledge financial support from the Research Council of Finland grant number 349002 held by L. Paunonen,  the Maupertuis programme of the Institut Français in Finland, and COST Action 
CA18232.}

\maketitle

\section{Introduction}\label{sec:intro}

Motivated by applications to boundary-coupled systems of partial differential equations (PDEs), we introduce new abstract tools for establishing  stability of coupled linear evolution equations. 
We focus on  the case of coupled systems consisting of two distinct parts, which in concrete situations will often be different in nature.
In particular, our results are applicable to systems in which  
a hyperbolic equation is coupled, along a shared boundary of the spatial domains, with a parabolic equation.
This scenario arises especially in \emph{thermoelastic systems} and in models describing \emph{fluid-structure interactions}.
More precisely, we consider a class of models which can be expressed
in the abstract form
\begin{equation}\label{eq:coupled_sys_intro}
\left\{\begin{aligned}
	  \dot{z}_1(t) &= L_1 z_1(t),  \qquad&t\ge0,\\
	  	  \dot{z}_2(t) &= L_2 z_2(t),  \qquad&t\ge0,\\
	  G_1 z_1(t) &=K_2z_2(t), &t\ge0,\\
	  G_2z_2(t) &= -K_1 z_1(t), &t\ge0,\\
	  z_1(0)&\in X_1,\ z_2(0)\in X_2.
	\end{aligned}\right.
\end{equation}
Here the first two lines define evolution equations on two Hilbert spaces $X_1$ and $X_2$, respectively, and the algebraic equations capture the interaction between the two systems. 
In the case of boundary-coupled PDEs, $L_1$ and $L_2$ are differential operators, and $G_j,K_j\colon D(L_j)\subseteq X_j\to U$, where $j=1,2$ and $U$ is another Hilbert space, 
are boundary trace operators representing the boundary coupling in the PDE model.

The standard approach to establishing well-posedness and stability of 
coupled PDE systems
 involves recasting the full  system as an abstract Cauchy problem on a suitably chosen Hilbert space, and using the theory of strongly continuous operator semigroups to establish  existence of solutions and study their asymptotic behaviour.
The required analysis is typically fairly involved, even for simple coupled  models.
In this paper, we take a \emph{divide-and-conquer} approach,
deducing stability of the coupled system  
 from properties of the two constituent parts, which in typical applications are either well understood or at least comparatively simple to analyse.

Our main assumptions are that the operators in~\eqref{eq:coupled_sys_intro} are chosen in such a way that the operator triples $(G_1,L_1,K_1)$ and $(G_2,L_2,K_2)$ form \emph{boundary nodes} in the sense defined in Section~\ref{sec:prelim}, and they are \emph{impedance passive} in the sense that
\begin{equation}
\label{eq:pass_intro}
\R\<L_jx,x\>_{X_j}\le\R\<G_jx,K_jx\>_U,\qquad x\in D(L_j),\ j=1,2.
\end{equation}
While these assumptions,  especially  impedance passivity, certainly place a restriction on  the class of boundary couplings that can be considered in~\eqref{eq:coupled_sys_intro}, condition~\eqref{eq:pass_intro} 
is in fact satisfied
 for a large class of physically relevant models.
These assumptions ensure that the system~\eqref{eq:coupled_sys_intro} is well posed, with solutions described by a strongly continuous semigroup of contractions.

Our main results relate the stability properties of the coupled system~\eqref{eq:coupled_sys_intro} 
 to those of the system
\begin{equation}\label{eq:damped_intro}
	\left\{\begin{aligned}
	  \dot{z}(t) &= L_1 z(t),  \qquad&t\ge0,\\
	  G_1 z(t) &= -K_1 z(t), &t\ge0,\\
	  z(0)&\in X_1.
	\end{aligned}\right.
\end{equation}
Observe that~\eqref{eq:damped_intro} 
contains the dynamics of the state $z_1$ appearing in
the first equation
 in~\eqref{eq:coupled_sys_intro} after decoupling from 
the state $z_2$.
The decoupling is completed by replacing the original coupling condition with the new algebraic constraint $z(t)\in\Ker(G_1+ K_1)$ for all $t\ge0$, which in typical PDE applications has the interpretation of \emph{boundary damping}.
In order to capture the contribution 
 of the second equation in~\eqref{eq:coupled_sys_intro} to the overall dynamics, we define $A_2\colon D(A_2)\subseteq X_2\to X_2$ to be the restriction of $L_2$ to $D(A_2)=\Ker G_2$.
Then $A_2$ is the 
generator of the semigroup describing the dynamics of $z_2$ in~\eqref{eq:coupled_sys_intro} with the homogeneous boundary condition $G_2z_2(t)=0$ for $t\geq 0$.
 We furthermore make use of the so-called \emph{transfer function} $P_2$ of the boundary node $(G_2,L_2,K_2)$, which maps $\gl\in \rho(A_2)$ to the bounded linear operator $P_2(\gl)$ on $U$ defined by  $P_2(\gl)u = K_2 x$ for $u\in U$, where
$x$ is the unique element $x\in D(L_2)$ such that
\[
(\gl-L_2)x=0 \qquad \mbox{and} \qquad G_2x = u.
\]
The following theorem contains a version of our first main result.
In Theorem~\ref{thm:res} below we shall present a more general result,  which does not require
 $ \norm{(is-A_2)\inv}$ or $ \norm{P_2(1+is)}$ to be uniformly bounded with respect to $s\in\RR$, and which can be used to obtain resolvent estimates on proper subsets of $i\RR$.
 We will also allow the space $U$ to be different for the two boundary nodes $(G_1,L_1,K_1)$ and $(G_2,L_2,K_2)$.

\begin{thm}\label{thm:main_intro}
Let $(G_1, L_1, K_1)$ and $(G_2,L_2,  K_2)$ be two impedance passive boundary nodes and let $A_0$ denote the restriction of $L_1$ to $\Ker (G_1+K_1)$.
 Suppose that $i\RR\subseteq \rho(A_0)\cap \rho(A_2)$
and let $M_0\colon \RR\to [r,\infty)$, for some $r>0$, be such that 
$$\|(is-A_0)\inv\|\le M_0(s),\qquad s\in \RR.$$
Suppose further that there exists a function
 $\eta\colon \RR\to(0,\infty)$ such that 
$$\R \iprod{P_2(is)u}{u}_U\geq \eta(s) \norm{u}^2, \qquad u\in U,~ s\in\RR$$
and that $ \norm{(is-A_2)\inv}$ and $ \norm{P_2(1+is)}$ are uniformly bounded with respect to $s\in\RR$.
Then~\eqref{eq:coupled_sys_intro} can be reformulated as an abstract Cauchy problem
on the Hilbert space $X=X_1\times X_2$,
 the corresponding system operator
 $A$ generates a contraction semigroup on $X$,  and satisfies
 $i\RR\subseteq\rho(A)$ and
$$
\|(is-A)\inv\|\lesssim  \frac{ M_0(s)}{\eta(s)},\qquad s\in \RR.
$$
\end{thm}

The resolvent estimate for $A$ presented in Theorem~\ref{thm:main_intro} may be combined with  results from the asymptotic theory of operator semigroups to describe the asymptotic behaviour of 
solutions to~\eqref{eq:coupled_sys_intro}. This is illustrated in the following corollary, which is a special case of  Corollary~\ref{cor:decay} below.

\begin{cor}
\label{cor:decay_intro}
Consider the setting of Theorem~\textup{\ref{thm:main_intro}}.
The contraction semigroup $\T$ generated by $A$ has the following properties:
\begin{enumerate}
\item[\textup{(a)}]  The semigroup  $\T$  is strongly stable.
\item[\textup{(b)}] If $ M_0(s)/\eta(s)$ is uniformly bounded with respect to $s\in\RR$, then 
 $\T$ is uniformly exponentially stable. 
\item[\textup{(c)}] If $M_0(s)/\eta(s)\lesssim 1+\abs{s}^\ga$ for some $\ga>0$ and all $s\in \RR$, then 
$$\norm{z_1(t)} + \norm{z_2(t)}=o(t^{-1/\alpha}),\qquad t\to\infty,$$
whenever
$z_1(0)\in D(L_1)$, 
$z_2(0)\in D(L_2)$ and the coupling conditions $G_1z_1(0)=K_2z_2(0)$ and $G_2z_2(0)=-K_1z_1(0)$ are satisfied.
\end{enumerate}
\end{cor}

We note that the conditions on the initial data $z_1(0)$ and $z_2(0)$ in part~(c) of Corollary~\ref{cor:decay_intro} are equivalent to $z_1$ and $z_2$ constituting a \emph{classical solution} of the abstract differential equation~\eqref{eq:coupled_sys_intro}, and the result corresponds to  \emph{polynomial stability} of the associated semigroup $\T$.
For many PDE systems of interest the quantity $\norm{z_1(\cdot)}^2+\norm{z_2(\cdot)}^2$ is proportional to the \emph{energy} of the solution of~\eqref{eq:coupled_sys_intro}, defined in a natural way, and thus Corollary~\ref{cor:decay_intro} immediately yields rates of energy decay in coupled PDE systems.
The resolvent bound $M_0$ for $A_0$ in Theorem~\ref{thm:main_intro} can often be taken directly from results found in the literature on, say, wave equations subject to boundary damping~\cite{KomZua90,Mar99}. Alternatively, one can often construct a suitable function $M_0$ by applying abstract methods such as those found in~\cite{AmmJel04,AmmTuc01,BouBou22,ChiPau23,DagZua06book}. 

In the final part of our paper we illustrate the power and versatility of our abstract results by using them to establish polynomial stability of selected PDE models, namely, networks of one-dimensional wave and heat equations and a multidimensional wave equation with an acoustic boundary condition.
 In the latter case, we apply variants of Theorem~\ref{thm:main_intro} and Corollary~\ref{cor:decay_intro} for systems in which the  boundary node $(G_1,L_1,K_1)$ is coupled with an infinite-dimensional linear system.

Well-posedness and stability of boundary-coupled PDE systems have been investigated extensively in the literature and, in particular, many authors have studied polynomial stability of coupled PDEs of mixed type, such as wave-heat or wave-beam systems.
For instance,  rational energy decay has been established by direct PDE methods for
multidimensional wave-heat models arising from fluid-structure interaction models~\cite{AvaLas16, AvaTri13,BatPau19, Duy07,ZhaZua07},
networks of one-dimensional PDE models~\cite{AmmShe18,Aug20a,HanZua16,LiWan24},
and PDE models with dynamic boundary conditions~\cite{AbbNic15b,MerNic17, RivAvi15,Rao97,Weh06}.
Results on  stability of coupled abstract differential equations may be found in~\cite{BenAmm16,Aug20b,BouBou21,DelPau23,FenGuo15,Nic23, Pau19}, and well-posedness of the particular system~\eqref{eq:coupled_sys_intro} has previously been established in~\cite{AalMal13}.
Our main results establish explicit resolvent bounds and decay rates for~\eqref{eq:coupled_sys_intro} with abstract boundary coupling. A particular strength of our results is that they allow us to exploit the relatively simple structure of the two constituent parts of the coupled system~\eqref{eq:coupled_sys_intro}.

Our notation and terminology are standard throughout. We implicitly take all of our  spaces to be complex, and we add subscripts to norms and inner products only in cases where there is risk of ambiguity. Given normed spaces $X$ and $Y$, we write $\B(X,Y)$ for the space of bounded linear operators from $X$ to $Y$, and we write $\B(X)$ for $\B(X,X)$. We denote the domain of a linear operator $L$  by $D(L)$, and we implicitly endow domains of linear operators with the graph norm. We denote the range and kernel of a linear operator $L$ by $\Ran L$ and $\Ker L$, respectively, and we write $\sigma(L)$ and $\rho(L)$ for the spectrum and the resolvent set of $L$. A linear operator $L\colon D(L)\subseteq X\to Y$ will be said to be \emph{bounded below} if there exists  a constant $c>0$ such that $\|Lx\|_Y\ge c\|x\|_X$ for all $x\in D(L)$. Furthermore, given a bounded linear operator $T\in\B(X)$ on a Hilbert space $X$, we denote by $\R T$ the self-adjoint operator $\frac12(T+T^*)$, and given two self-adjoint bounded linear operators $S,T\in\B(X)$ on a Hilbert space $X$ we write $T\ge S$ to mean that $\langle Tx-Sx,x\rangle\ge0$ for all $x\in X$. If $p$ and $q$ are two real-valued quantities we write $p\lesssim q$ to express that $p\le Cq$ for some constant $C>0$ which is independent of all parameters that are free to vary in a given situation. We shall also make use of standard `big-O' and `little-o' notation.

\section{Preliminary results on boundary nodes}\label{sec:prelim}

Let  $X$ and $U$ be Hilbert spaces, and let $L\colon D(L)\subseteq X\to X$, $G, K\colon D(L)\subseteq X\to U$ be linear operators. The triple $(G,L,K)$
is said to be an \emph{(internally well-posed) boundary node} (on the Hilbert spaces $(U,X,U)$) if 
\begin{itemize}
\item[\textup{(i)}]  $G,K\in\B(D(L),U)$;
\item[\textup{(ii)}] the restriction $A\colon\Ker G\subseteq X\to X $ of $L$ to $\Ker G$ generates a $C_0$-semigroup on $X$;
\item[\textup{(iii)}]  the operator $G$ has a right-inverse, which is to say there exists $G^r\in\B(U,D(L))$ such that $GG^r=I$.
\end{itemize}
The boundary node is said to be \emph{impedance passive} if
\begin{equation}\label{eq:pass}
\R\<Lx,x\>_X\le\R\<Gx,Kx\>_U,\qquad x\in D(L),
\end{equation}
and in this case it follows from the Lumer--Phillips theorem \cite[Thm.~II.3.15]{EngNag00book} that the $C_0$-semigroup generated by $A$ is contractive.
Note further that for any boundary control system the domain $D(L)$ of $L$ is densely (and continuously) embedded in $X$. Furthermore, the existence of a right-inverse for $G$ is guaranteed if, for instance, $U$ is finite-dimensional and $G$ is surjective. A boundary node may be associated with the \emph{abstract boundary control system}~\cite{CheMor03, MalSta06, Sal87a} given by
\begin{equation}\label{eq:BCSgen}
	\left\{\begin{aligned}
	  \dot{z}(t) &= L z(t),  \qquad&t\ge0,\\
	  G z(t) &= u(t), &t\ge0,\\
	  y(t) &= K z(t), &t\ge0,\\
	  z(0)&=z_0\in X.
	\end{aligned}\right.
\end{equation}

\begin{rem}\label{rem:Riesz}
As noted for instance in~\cite{MalSta07}, it is possible to consider boundary nodes also in the ostensibly more general setting where $K$ maps not into the codomain $U$ of $G$ but instead into another Hilbert space $Y$. In this case it is customary to assume either that $U$ is the space of bounded conjugate-linear functionals on $Y$ or, vice versa, that $Y$ is the space of bounded conjugate-linear functionals on $U$, so that the inner product of $U$ in~\eqref{eq:pass} can be replaced by the dual pairing between $U$ and $Y$. The Riesz--Fréchet theorem ensures the existence of a unitary map $\Phi\colon U\to Y$, so we may return to the setting considered initially by replacing either  $K$ by $\Phi^* K$ or $G$ by $\Phi G$, resulting in a boundary node on the spaces $(Y,X,Y)$. 
\end{rem}

\begin{rem}\label{rem:equiv}
Our definition of an (internally well-posed) boundary node is ostensibly different from that given in  \cite[Def.~2.2]{AalMal13} and \cite[Def.~2.2]{MalSta07}, but in fact the two concepts coincide. More precisely, $(G,L,K)$ is an internally well-posed boundary node on $(U, X, U )$ if and only if $(G, L, K)$ is an internally well-posed boundary node in the sense of  \cite[Def.~2.2]{AalMal13} and \cite[Def.~2.2]{MalSta07} and a \emph{strong colligation} in the sense of \cite[Def.~4.4]{MalSta07}. Moreover, our definition of impedance passivity of the boundary node $(G,L,K)$ is equivalent to the definition in  \cite[Def.~2.3]{AalMal13} and \cite[Def.~3.2]{MalSta07}.
\end{rem}

We omit the proofs of the claims made in Remark~\ref{rem:equiv}, except to point out that whereas closedness of the operator $L$, or equivalently completeness of $D(L)$, is \emph{assumed} in \cite{AalMal13, MalSta07} for us it is a consequence of the definition of a boundary node, as the following simple result shows.

\begin{lem}\label{lem:closed}
If $(G,L,K)$ is a boundary node then $L$ is a closed operator.
\end{lem}

\begin{proof}
Let $(x_n)_{n\ge1}$ be a sequence in $D(L)$ and suppose that $x,y\in X$ are such that  $x_n\to x$  and $Lx_n\to y$ as $n\to\infty$.  Since $G\in\B(D(L),U)$, the sequence $(Gx_n)_{n\ge1}$ is a Cauchy sequence in $U$, so by completeness there exists $u\in U$ such that $Gx_n\to u$ as $n\to\infty$. Let $y_n=x_n-G^rGx_n\in\Ker G$ for $n\ge1$, where $G^r\in\B(U,D(L))$ is a right-inverse of $G$. Then $y_n\to x-G^ru$ as $n\to\infty$.  Since $LG^r\in\B(U,X)$, the restriction $A$ of $L$ to $\Ker G$ satisfies
$$Ay_n=Lx_n-LG^rGx_n\to y-LG^ru,\qquad n\to\infty.$$
 The operator $A$ is  by assumption the generator of a $C_0$-semigroup on $X$, and in particular closed, so $x-G^ru\in D(A)=\Ker G\subseteq D(L)$ and $A(x-G^ru)=y-LG^ru$. Hence $x\in D(L)$ and 
$Lx=A(x-G^ru)+LG^ru=y,$
and it follows that $L$ is closed.
\end{proof}

Let $(G,L,K)$ be a boundary node on $(U,X,U)$, and let $A$ be the restriction of $L$ to $\Ker G$. We define the \emph{transfer functions} $H\colon \rho(A)\to\B(U,X)$ and $P\colon\rho(A)\to\B(U)$ by $H(\lambda)u=x$ and $P(\lambda)u=Kx$ for all $\lambda\in\rho(A)$ and $u\in U$, where $x\in D(L)$ is the solution of the abstract elliptic problem
$$(\lambda-L)x=0,\qquad Gx=u.$$
The existence of $H$ and $P$ follow from~\cite[Rem.~10.1.5 \& Prop.~10.1.2]{TucWei09book}.
In defining the transfer functions of a boundary node, we do not distinguish notationally between $H$ and $P$ and their analytic extensions to larger domains. Our first preparatory result establishes various identities and estimates, primarily for the transfer functions $H$ and $P$. Parts of the proof rely on the theory of \emph{system nodes}~\cite[Sect~4.7]{Sta05book}.

\begin{prp}
\label{prp:HPproperties}
Assume that $(G,L,K)$ is a boundary node on $(U,X,U)$ and let $A$ be the restriction of $L$ to $\Ker G$. 
\begin{itemize}
\item[\textup{(a)}] 
The functions $H$ and $P$ are analytic,
and for $\gl\in\rho(A)$ the following hold:
\begin{itemize}
\item[\textup{(i)}]
$L H(\lambda)=\lambda H(\lambda)$, 
  $G H(\gl)=I$ and $K H(\gl)=P(\gl)$;
\item[\textup{(ii)}]  $H(\gl)\in \B(U,D(L))$;
\item[\textup{(iii)}] $(\gl -A)\inv (\gl-L)x=x-H(\gl)G x$ for all  $x\in D(L)$, and in particular $H(\gl)G x=x$ for all $x\in\Ker (\gl-L)$.
\end{itemize}
\item[\textup{(b)}]
Assume in addition that the boundary node $(G,L,K)$ is impedance passive. Then $\R P(\gl)\geq 0$ for all $\gl\in  \rho(A)\cap\overline{\C_+}$, and 
\begin{equation}\label{eq:square_est}
\|{H(\gl)}\|^2 \leq \frac{1}{\R\gl}\|{P(\gl)}\|, \qquad \|{K(\gl-A)\inv}\|^2 \leq \frac{1}{\R\gl}\|{P(\gl)}\|
\end{equation}
for all $\gl\in\C_+$. Moreover, if $\gl\in\rho(A)\cap\overline{\C_+}$ then
$$
\begin{aligned}
\|{K(\gl-A)\inv}\|& \leq \|K(1+\gl-A)\inv\|(1+\|(\gl-A)\inv\|),\\
\|{H(\gl)} \|&\leq \|{H(1+\gl)}\|(1+\|{(\gl-A)\inv}\|),\\
\|{P(\gl)} \|&\leq \|{P(1+\gl)}\|(2+\|{(\gl-A)\inv}\|).
\end{aligned}
$$
\end{itemize}
\end{prp}

\begin{proof}
Part (a)(i) follows immediately from the definitions of $H$ and $P$, and part~(a)(ii) follows from the first part of~(a)(i). In order to prove part~(a)(iii) let $\lambda\in\rho(A)$, $x\in D(L)$ and let $y=(\gl -A)\inv (\gl-L)x$. Then $y\in D(A)=\Ker G$ and $(\gl-A)y=(\gl-L)x$. Hence $y-x\in\Ker(\gl-L)$ and $G(y-x)=-Gx$, so by definition of $H$ we have $y-x=-H(\gl)Gx$, as required. 
Since $G$ is surjective by assumption, 
 part (iii) and the resolvent identity can be used to show that $H(\gl)-H(\mu) = (\mu-\gl)(\gl-A)\inv H(\mu)$ for all $\gl,\mu\in\rho(A)$.
Analyticity of $H$ and $P$ follows easily.

Now suppose that the boundary node $(G,L,K)$ is impedance passive and let $\gl\in  \rho(A)\cap\overline{\C_+}$, $u\in U$ and  $x=H(\lambda)u$. Then $x\in \Ker(\gl-L)$, $Gx=u$ and $Kx=P(\gl)u$, so
$$\R \<P(\gl)u,u\>=\R \<Kx,Gx\>\ge\R \<Lx,x\>=\R\lambda\,\|x\|^2.$$ 
Thus $\R P(\gl)\ge0$ and $ \R\lambda\,\norm{H(\gl)u}^2\leq  \norm{P(\gl)}\norm{u}^2$, giving both the first claim in (b) and the first inequality in~\eqref{eq:square_est}.
By~\cite[Thm.~2.3]{MalSta06} the boundary node $(G,L,K)$ defines a system node 
\begin{equation}\label{eq:S_node}
S=\begin{pmatrix}
A\&B\\ C\&D
\end{pmatrix}
=\begin{pmatrix}
L\\ K
\end{pmatrix}
\begin{pmatrix}
I\\G
\end{pmatrix}\inv,
\qquad D(S)=\Ran \begin{pmatrix}
I\\G
\end{pmatrix},
\end{equation}
and the transfer function of the system node $S$ coincides on $\rho(A)\cap\overline{\C_+}$ with $P$. Here the semigroup generator associated with  $S$ is precisely $A$, the output operator  $C$ of $S$ is the restriction of $K$ to $D(A)=\Ker G$ and the input operator $B$ satisfies $H(\gl)=(\gl-A)\inv B$ for all $\gl\in\rho(A)\cap\overline{\C_+}$, by~\cite[Prop.~10.1.2 \& Rem.~10.1.5]{TucWei09book}. 
 Here and elsewhere we do not distinguish notationally between $A$ and its extension to an element of $\B(X,X_{-1})$, where $X_{-1}$ is the completion of $X$ with respect to the norm defined by $\|x\|_{-1}=\|(\gl_0-A)\inv x\|$ for some choice of $\gl_0\in\rho(A)$.
Since $(G,L,K)$ is impedance passive, it follows from~\cite[Thm.~2.3]{MalSta06} and~\cite[Thm.~4.2]{Sta02}
that the system node $S$ is impedance passive  in the sense of~\cite[Def.~4.1]{Sta02}, which is to say that
$$\R\<{Ax+Bu,x}\>\leq \R \left\<{C\&D\begin{pmatrix}
x\\u
\end{pmatrix}
,u}\right\>,\qquad
\begin{pmatrix}
x\\u
\end{pmatrix}
\in D(S).$$
Let $\gl\in\C_+$ and 
$x=(\gl-A)\inv Bu$. Then $Ax+Bu = \gl x\in X$ and $(C\&D)(x,u)^\top = P(\gl)u$, and
impedance passivity of $S$ gives
$$
\begin{aligned}
\R \gl \,\|{(\gl-A)\inv Bu}\|^2
 = \R \<{\gl x,x}\>
 = \R \<{Ax+Bu,x}\>
\le\R \<{P(\gl)u,u}\>.
\end{aligned}
$$
Thus
 $\|{(\gl-A)\inv B}\|^2\leq (\R \gl)\inv\|{P(\gl)}\|$ for $\gl\in \C_+$,
and in fact the same estimate holds for impedance passive system nodes which do not necessarily arise from boundary control systems.
 By~\cite[Thm.~6.2.14]{Sta05book} the dual system node has  domain 
$$D(S^\ast) =\left \{\begin{pmatrix}
x\\u
\end{pmatrix}\in X\times U:A^\ast x+C^\ast u\in X\right\},$$ 
and its generating operators are $A^\ast$, $C^\ast$, and $B^\ast$, and its transfer function on $\C_+$ is given by $\gl\mapsto P(\overline{\gl})^\ast$. Here $C^*$ is the adjoint of $C\in\B(D(A),U)$. Furthermore, $S$ is impedance passive if and only if $S^\ast$ is impedance passive~\cite[Cor.~4.5]{Sta02}. 
Hence  $\|{(\overline{\gl}-A^\ast)\inv C^\ast}\|^2\leq (\R \gl)\inv\|{P(\gl)^\ast}\|$ for all $\gl\in\C_+$ 
by the above argument.
Recalling that  $C$ coincides with the restriction of $K$ to ${D(A)}$, we deduce that if $\gl\in\C_+$ then $$(K(\gl-A)\inv)^\ast = (C(\gl-A)\inv)^\ast = (\overline{\gl}-A^\ast)\inv C^\ast,$$ 
where $(\overline{\gl}-A^\ast)\inv $ is the adjoint of $(\gl-A)\inv\in\B(X,D(A))$.
 It follows that $\|{K(\gl-A)\inv}\|^2\le  (\R\gl)\inv\|{P(\gl)}\|$. 

Finally, let $\gl\in\rho(A)\cap\overline{\C_+}$. Then $(\gl-A)\inv=(1+\gl-A)\inv(I+(\gl-A)\inv)$
by the resolvent identity, and the first of the final three estimates follows at once.
If $u\in U$ and $x = H(1+\gl)u\in D(L)$, then
$(1+\gl-L)x=0$ and $Gx=u$.
 By part~(a)(iii) we have $x=-(\gl-A)\inv x+H(\gl)u.$ Hence $H(\gl) = (I+(\gl-A)\inv)H(1+\gl)$. This identity immediately gives the second estimate, and also implies the final estimate  using part~(a)(i),  \eqref{eq:square_est} and the previous estimate for $\|K(\gl-A)\inv\|$.
\end{proof}

\begin{rem}
If we define $L\colon D(L)\subseteq X\to X$ by $L=\diag(L_1,L_2)$ with domain $D(L)=D(L_1)\times D(L_2)$ and $G,K\colon D(L)\subseteq X\to U=U_1\times U_2$ by  $G=\diag(G_1,G_2)$ and $K=\diag(K_1,K_2)$, then it is easy to see that $(G,L,K)$ is an impedance passive boundary node on $(U,X,U)$, and the associated transfer function $P$ satisfies $P(\gl)=\diag(P_1(\gl),P_2(\gl))$ for all $\gl\in\rho(A_1)\cap\rho(A_2)\cap\overline{\C_+}$.
\end{rem}

In what follows we shall  consider the restriction of $L$ to $\Ker(G+J^*QJK)$, where $(G,L,K)$ is an impedance passive boundary node on $(U,X,U)$,   $Q$ is a bounded linear operator on a Hilbert space $V$ and $J\in\B(U,V)$. We denote this operator by $A_Q$, noting that  $A_0$ coincides with the operator $A$. The operator $A_Q$ is associated with the system 
\begin{equation}\label{eq:}
	\left\{\begin{aligned}
	  \dot{z}(t) &= L z(t),  \qquad&t\ge0,\\
	  G z(t) &= -J^*QJKz(t), &t\ge0,\\
	  z(0)&=z_0\in X.
	\end{aligned}\right.
\end{equation}
Here the second line represents \emph{boundary damping} when $\R Q\ge0$.

\begin{prp}
\label{prp:Q}
Let $(G,L,K)$ be an impedance passive boundary node on $(U,X,U)$, $J\in\B(U,V)$ and suppose that $Q\in\B(V)$ satisfies $\R Q\ge cI$ for some $c>0$. 
\begin{itemize}
\item[\textup{(a)}]
 The triple $(G+J^*QJK,L,K)$ is an impedance passive boundary node on $(U,X,U)$ and $\rho(A)\cap \overline{\C_+}\subseteq \rho(A_Q)\cap \overline{\C_+}$. Moreover, 
  \begin{equation}\label{eq:diss_est}
  \R\< A_Qx,x\>\le -c\|JKx\|^2,\qquad x\in D(A_Q).
  \end{equation} 
 \item[\textup{(b)}] Let $H_Q\colon\rho(A_Q)\to\B(U,X)$ and  $P_Q\colon\rho(A_Q)\to\B(U)$ denote the transfer functions associated with  $(G+J^*QJK,L,K)$.  If $\gl\in\rho(A)\cap \overline{\C_+}$ then $I+J^*QJP(\gl)$ is invertible and 
$$
\begin{aligned}
H_Q(\gl) &= H(\gl) (I+J^*QJP(\gl))\inv,\\ 
P_Q(\gl) &= P(\gl)(I+J^*QJP(\gl))\inv .
\end{aligned}
$$
 
\item[\textup{(c)}] If $\lambda\in\rho(A_Q)\cap\overline{\C_+}$ then 
$$
\begin{aligned}
\|JK(\gl-A_Q)\inv\|^2 &\le c\inv \|(\gl-A_Q)\inv\|,\\
\|H_Q(\gl)J^*\|^2 &\le c\inv \|(\gl-A_Q)\inv\|,\\
\|JP_Q(\gl)J^*\|&\le c\inv.
\end{aligned}
$$ 

\end{itemize}
\end{prp}

\begin{proof}
We begin by proving the inclusion $\rho(A)\cap \overline{\C_+}\subseteq \rho(A_Q)\cap \overline{\C_+}$ in~(a).
To this end, let $\gl\in\rho(A)\cap \overline{\C_+}$. 
Then $\R P(\gl)\geq 0$ by  Proposition~\ref{prp:HPproperties}(b), and thus also $\R JP(\gl)J^*\geq 0$.
Since $\R Q\geq cI$,  $Q$ is invertible and $\R Q\inv \geq c \|{Q}\|^{-2} I$ by~\cite[Lem.~A.1]{Pau19}.
Thus $I+QJP(\gl)J^*=Q(Q\inv + JP(\gl)J^*)$. But 
$$\R (Q\inv + JP(\gl)J^*) \geq \R Q\inv \geq c \|{Q}\|^{-2} I,$$ 
and hence $I+ QJP(\gl)J^*$ is  invertible, which finally implies that $I+ J^*QJP(\gl)$, too, is invertible. 
Let  
\begin{equation}\label{eq:AQ_inv}
R_\gl=(\gl-A)\inv - H(\gl)(I+J^*QJP(\gl))\inv J^*QJK(\gl-A)\inv,
\end{equation}
noting that $R_\gl\in \B(X)$ and $\Ran R_\gl\subseteq D(L)$. By Proposition~\ref{prp:HPproperties}(a) we have
 $(G+J^*QJK)H(\gl)=I+J^*QJP(\gl)$, and hence  a straightforward calculation gives $\Ran R_\gl\subseteq D(A_Q)=\Ker(G+J^*QJK)$. Using the fact that $(\gl-L)H(\gl)=0$, again by Proposition~\ref{prp:HPproperties}(a), we see that $(\gl-A_Q)R_\gl=I$. Now let $x\in D(A_Q)$. Then $Gx=-J^*QJKx$ and hence 
 $$(\gl-A)^{-1}(\gl-L)x=x+H(\gl)J^*QJKx$$ by Proposition~\ref{prp:HPproperties}(a). Another straightforward calculation now shows that $R_\gl(\gl-A_Q)x=x$. Hence $R_\gl$ is the inverse of $\gl-A_Q$, so $\gl\in\rho(A_Q)\cap\overline{\C_+}$. If $x\in D(L)$, then  by impedance passivity of $(G,L,K)$ we have
 \begin{equation}\label{eq:pass_est}
\begin{aligned}
 \R\<Lx,x\>&\le \R\<Gx,Kx\>\\&=\R\<Gx+J^*QJKx,Kx\>-\R\<J^*QJKx,Kx\>\\&\le \R\<Gx+J^*QJKx,Kx\>-c\|JKx\|^2.
 \end{aligned}
  \end{equation}
 This establishes the estimate~\eqref{eq:diss_est} and in particular shows that $A_Q$ is dissipative. Since $1\in\rho(A)\cap\C_+\subseteq\rho(A_Q)\cap\C_+$, the operator $A_Q$ is maximally dissipative, so by the Lumer--Phillips theorem it generates a $C_0$-semigroup of contractions. Next we use Proposition~\ref{prp:HPproperties}(a) to observe that $(G+J^*QJK)H(1)=I+J^*QJP(1)$ and that $H(1)\in\B(U,D(L))$. It follows that $H(1)(I+J^*QJP(1))\inv\in\B(U,D(L))$ is a right-inverse of $G+J^*QJK$. Hence $(G+J^*QJK,L,K)$ is a boundary node on $(U,X,U)$, and by~\eqref{eq:pass_est} it is impedance passive.
 
 In order to prove~(b), let $\gl\in\rho(A)\cap \overline{\C_+}$ and $u\in U$. Then $x=H_Q(\gl)u$ satisfies $(\gl-L)x=0$ and $Gx+J^*QJKx=u$. Proposition~\ref{prp:HPproperties}(a) implies that 
 \begin{equation}\label{eq:HQ}
 x=H(\gl)Gx=H(\gl)u-H(\gl)J^*QJKx,
 \end{equation}
  and hence  $Kx=P(\gl)u-P(\gl)J^*QJKx.$ By the proof of part~(a) the operator $I+J^*QJP(\gl)$ is invertible, and hence $I+P(\gl)J^*QJ$, too, is invertible and 
 $$(I+P(\gl)J^*QJ)\inv =I-P(\gl)(I+J^*QJP(\gl))\inv J^*QJ.$$ It follows that 
 $$
Kx=(I+P(\gl)J^*QJ)\inv P(\gl)u=P(\gl)(I+J^*QJP(\gl))\inv u,$$
 giving $P_Q(\gl)=P(\gl)(I+J^*QJP(\gl))\inv$, as required. Together with~\eqref{eq:HQ} this implies that 
 $$H_Q(\gl)=H(\gl)(I-J^*QJP(\gl)(I+J^*QJP(\gl))\inv)=H(\gl)(I+J^*QJP(\gl))\inv,$$
thus completing the proof of part~(b).

We now turn to part~(c), starting with the final estimate. Let $\gl\in\rho(A_Q)\cap\overline{\C_+}$ and $u\in U$, and let $x=H_Q(\gl)J^*u$. Then $(\gl-L)x=0$ and $(G+J^*QJK)x=J^*u$. Moreover, $JP_Q(\gl)J^*u=JKx$. Using impedance passivity of $(G,L,K)$ we find that
$$\begin{aligned}
c\|JP_Q(\gl)J^*u\|^2&=c\R\<JKx,JKx\>\le\R\<QJKx,JKx\>\\&=\R\<J^*u,Kx\>-\R\<Gx,Kx\>\le \R\<u,JKx\>-\R\<Lx,x\>\\&= \R\<u,JP_Q(\gl)J^*u\>-\R\gl\,\|x\|^2\le\|u\|\|JP_Q(\gl)J^*u\|,
\end{aligned}
$$
and hence $\|JP_Q(\gl)J^*\|\le c\inv$, as required. 
In order to prove the remaining two estimates in~(c), we consider the system nodes $S$ and $S_Q$ associated with the boundary nodes $(G,L,K)$ and $(G+J^*QJK,L,K)$, respectively. By~\cite[Thm.~2.3]{MalSta06} these system nodes are given by~\eqref{eq:S_node} and by
$$S_Q=\begin{pmatrix}
A_Q\&B_Q\\ C_Q\&D_Q
\end{pmatrix}
=\begin{pmatrix}
L\\ K
\end{pmatrix}
\begin{pmatrix}
I\\G+J^*QJK
\end{pmatrix}\inv$$
with domain 
$$D(S_Q)=\Ran \begin{pmatrix}
I\\G+J^*QJK
\end{pmatrix}.
$$
The output operator $C_Q$ of $S_Q$ is the restriction of $K$ to $D(A_Q)=\Ker(G+J^*QJK)$, $H_Q(\gl)=(\gl-A_Q)\inv B_Q$ for all $\gl\in\rho(A_Q)\cap\overline{\C_+}$ and the transfer function of the system node $S_Q$ coincides on $\rho(A_Q)\cap\overline{\C_+}$ with $P_Q$. Furthermore, impedance passivity of the boundary nodes $(G,L,K)$ and $(G+J^*QJK,L,K)$ implies that the system nodes $S$ and $S_Q$ are impedance passive in the sense of~\cite[Def.~4.1]{Sta02}. Since $1\in\rho(A)\cap\overline{\C_+}$ we know that $I+J^*QJP(1)$ and  $I+P(1)J^*QJ$ are invertible. Moreover, by~\eqref{eq:AQ_inv} we have
$$(I-A_Q)\inv=(I - H(1)(I+J^*QJP(1))\inv J^*QJK)(I-A)\inv,$$
and a straightforward computation using Proposition~\ref{prp:HPproperties}(a) shows that
$$C_Q(I-A_Q)\inv=K(I-A_Q)\inv=(I+P(1)J^*QJ)\inv K(I-A)\inv.$$
Hence using Proposition~\ref{prp:Q}(b) we obtain
$$\begin{aligned}
\!\begin{pmatrix}
(I - H(1)(I+J^*QJP(1))\inv J^*QJK)(I-A)\inv& H(1)(I+J^*QJP(1))\inv\\ 
(I+P(1)J^*QJ)\inv K(I-A)\inv& P(1)(I+J^*QJP(1))\inv
\end{pmatrix}
\\=
\begin{pmatrix}
(I-A_Q)\inv& (I-A_Q)\inv B_Q\\ 
C_Q (I-A_Q)\inv& P_Q(1)
\end{pmatrix}.
\end{aligned}$$
It follows from~\cite[Thm.~7.4.7(ix)]{Sta05book} that $-J^*QJ$ is an admissible output feedback operator for $S$ in the sense of~\cite[Def.~7.4.2]{Sta05book}, and furthermore $S_Q$ is precisely the closed-loop system node resulting from $S$ by applying this feedback.  Now let $\gl\in \rho(A_Q)\cap\overline{\C_+}$, let $y\in X$ and define $x\in D(A_Q)$ by $x=(\gl-A_Q)\inv y$. Moreover, let $u=-J^*QJC_Qx\in U$. Then \cite[Thm.~7.4.5]{Sta05book} gives
$$Ax+Bu=Ax-BJ^*QJC_Qx=A_Qx\in X$$ 
and $(C\& D)(x,u)^\top=C_Qx$. 
Using~\cite[Thm.~4.2(iii)]{Sta02} and impedance passivity of $S$ we deduce that
$$\begin{aligned}
c\|JC_Qx\|^2&\le\R\<JC_Qx,QJC_Qx\>
=\R\left\<C\&D \begin{pmatrix}
x\\u
\end{pmatrix},J^*QJC_Qx\right\>\\&=-\R\left\<C\&D \begin{pmatrix}
x\\u
\end{pmatrix},u\right\>
\le-\R\<Ax+Bu,x\>\\&\le\R\<(\gl-A_Q)x,x\>\le\|(\gl-A_Q)\inv\|\|y\|^2,
\end{aligned}$$
and hence 
\begin{equation}\label{eq:Kres}
 \|JK(\gl-A_Q)\inv\|^2=\|JC_Q(\gl-A_Q)\inv\|^2 \le c\inv \|(\gl-A_Q)\inv\|,
 \end{equation}
as required.
Our argument shows
that the second inequality in~\eqref{eq:Kres} is also valid 
for any system node $S_Q$ which is 
obtained by feedback from an impedance passive system node $S$ with 
 an admissible feedback operator of the form $-J^\ast QJ$.
 In order to prove the final estimate in~(c), we consider the dual system node $(S_Q)^*$, which by~\cite[Thm.~7.6.1(ii)]{Sta05book}  coincides with the closed-loop system node obtained from the impedance passive system node $S^*$ by applying the admissible output feedback operator $(-J^*QJ)^*=-J^*Q^*J$. Furthermore, by~\cite[Lem.~6.2.14]{Sta05book} its generating operators are $(A_Q)^*$, $(C_Q)^*$ and $(B_Q)^*$. Let $\gl\in\rho(A_Q)\cap\overline{\C_+}$. Then $\overline{\gl}\in\rho((A_Q)^*)\cap\overline{\C_+}$ and, since $\R Q^*=\R Q\ge cI$,
the inequality in 
 \eqref{eq:Kres} gives %[OR ``the second inequality in''??]
$$\|J(B_Q)^*(\overline{\gl}-(A_Q)^*)\inv\|^2 \le c\inv \|(\overline{\gl}-(A_Q)^*)\inv\|=c\inv\|({\gl}-A_Q)\inv\|.$$ 
Noting that
$$\|H_Q(\gl)J^*\|=\|(\gl-A_Q)\inv B_QJ^*\|=\|J(B_Q)^*(\overline{\gl}-(A_Q)^*)\inv\| ,$$
where  $(\overline{\gl}-(A_Q)^*)\inv$ is the adjoint of $(\gl-A_Q)\inv\in\B(X,D(A_Q))$ and $(B_Q)^*\in\B(D((A_Q)^*),U)$, we deduce that $\|H_Q(\gl)J^*\|^2\le c\inv\|({\gl}-A_Q)\inv\|$, thus completing the proof.
\end{proof}

\begin{rem}
\label{rem:Qweakerconds}
As the proof of part~(a) makes clear, $(G+J^*QJK,L,K)$ is an impedance passive boundary node even if we replace the assumption that $\R Q\ge cI$ for some $c>0$ by the weaker conditions that $\R Q\ge0$ and $I+J^*QJP(\gl)$ is invertible for some $\gl\in\rho(A)\cap\overline{\C_+}$. 
By~\cite[Lem.~A.1(d)]{Pau19} this is in particular true whenever $Q$ is self-adjoint and $Q\geq 0$. 
\end{rem}

We conclude this preliminary section with a result providing sufficient conditions under which the roles of the input and output of an impedance passive boundary node can be interchanged. Note that if $(G, L, K)$ is an impedance passive boundary node and if $P(\gl)$ is invertible for some $\gl\in\rho(A)\cap\overline{\C_+}$, then $P(\gl)$ is  invertible for all $\gl\in{\C_+}$ by~\cite[Cor.~4.3]{Log20}.

\begin{prp}\label{prp:swap}
Let $(G,L,K)$ be an impedance passive boundary node on $(U,X,U)$ and suppose that $P(\gl)$ is invertible for some $\gl\in\rho(A)\cap\overline{\C_+}$.  Then $(K,L,G)$ is an impedance passive boundary node on $(U,X,U)$ and $\gl\in\rho(B)$, where $B$ denotes the restriction of $L$ to $\Ker K$. Furthermore, if $P_*$ denotes the counterpart of $P$ for the boundary node $(K,L,G)$ then 
$$\{\gl\in\rho(A)\cap i\RR:P(\gl)\mbox{ \rm is invertible}\}=\{\gl\in\rho(B)\cap i\RR:P_*(\gl)\mbox{ \rm is invertible}\},$$
and $P_*(\gl)=P(\gl)\inv$ for all $\gl \in\rho(A)\cap\overline{\C_+}$ such that $P(\gl)$ is invertible.
\end{prp}

\begin{proof}
In order to verify that $(K,L,G)$ is an impedance passive boundary node we only need to show that $B$ generates a $C_0$-semigroup on $X$ and that $K$ has a right-inverse, since~\eqref{eq:pass} is symmetric in $G$ and $K$. Let   $\gl\in\rho(A)\cap\overline{\C_+}$ be such that $P(\gl)$ is invertible and set $K^r=H(\gl)P(\gl)\inv$. By Proposition~\ref{prp:HPproperties}(a) we have $K^r\in\B(U,D(L))$ and $KK^r=I$, so $K^r$ is a right-inverse of $K$. Moreover, by~\eqref{eq:pass} we have $\R\< Bx,x\>\le0$ for all $x\in \Ker K$, so $B$ is dissipative. We now show that $\gl\in \rho(B)$. Since $B$ is the restriction of  $L$ to $\Ker K$, and $L$ is closed by Lemma~\ref{lem:closed}, it is straightforward to show that $B$ is closed. Hence it suffices to prove that $\gl-B$ is bijective. Suppose therefore that $x\in\Ker(\gl-B)$. Then $x\in \Ker K\cap\Ker (\gl-L)$, so using Proposition~\ref{prp:HPproperties}(a) we obtain
$$0=Kx=KH(\gl)Gx=P(\gl)Gx.$$ 
Since $P(\gl)$ is assumed to be invertible we deduce that $Gx=0$, and hence $x\in \Ker(\gl-A)=\{0\}$, so $\gl-B$ is injective. Now let $y\in X$ be arbitrary, and define $x\in X$ by
$$x=(\gl-A)\inv y-H(\gl)P(\gl)\inv K(\gl-A)\inv y.$$
Straightforward calculations using Proposition~\ref{prp:HPproperties}(a) show that $x\in \Ker K$ and that $(\gl-L)x=y$, so $\gl-B$ is surjective. It follows from the Lumer--Phillips theorem that $B$ is the generator of a $C_0$-semigroup on $X$, and hence $(K,L,G)$ is an impedance passive boundary node, as required. Furthermore, for $\gl\in\C_+$ the operator $P(\gl)$ is invertible and the definitions of $P$ and $P_*$ imply that $P_*(\gl)=P(\gl)\inv$. In particular, $P_*(\gl)$ is invertible.  If $\gl\in \rho(A)\cap i\RR$ is such that $P(\gl)$ is invertible then, as was shown above, $\gl\in\rho(B)$ and by continuity we again have $P_*(\gl)=P(\gl)\inv$. This also proves that
$$\{\gl\in\rho(A)\cap i\RR:P(\gl)\mbox{ \rm is invertible}\}\subseteq\{\gl\in\rho(B)\cap i\RR:P_*(\gl)\mbox{ \rm is invertible}\},$$
The reverse inclusion follows on interchanging the roles of $G$ and $K$.
\end{proof}

\begin{rem}
\label{rem:BCSIOswapPartial}
It is also possible to interchange the roles of only parts of the inputs and outputs of  an impedance passive boundary node $(G,L,K)$ under a strictly weaker condition. To consider this situation let $U=U_1\times U_2$ and 
$$
G = \begin{pmatrix}
G_1\\G_2
\end{pmatrix}, \qquad 
K = \begin{pmatrix}K_1\\K_2\end{pmatrix}, \qquad 
G' = \begin{pmatrix}K_1\\G_2\end{pmatrix}, \qquad 
K' = \begin{pmatrix}G_1\\K_2\end{pmatrix}.
$$
Furthermore, let $L_1$ be the restriction of $L$ to $\Ker G_2$ and let $P_1$ be the transfer function of the impedance passive boundary node $(G_1,L_1,K_1)$.
Then $(G',L,K')$ is an impedance passive boundary node provided  $P_1(\gl)$ is invertible for some $\gl\in\rho(A)\cap \overline{\C_+}$.
We leave the straightforward modification of the proof of Proposition~\ref{prp:swap} to the reader.
\end{rem}

\section{Abstract coupled systems}\label{sec:coupled}

In this section we study well-posedness and resolvent estimates for coupled systems. 
We begin by considering the case of two coupled boundary control systems; later in the section, we will consider the case of a boundary control system coupled with a linear system. 

Given Hilbert spaces $X_1,$ $X_2$ and $U_1,U_2$, let $(G_1, L_1, K_1)$ and $( G_2, L_2,  K_2)$ be two impedance passive boundary nodes on $(U_1,X_1,U_1)$ and $(U_2,X_2,U_2)$, respectively, and let $J\in\B(U_1,U_2)$. We consider the coupled system
\begin{equation}\label{eq:coupled_sys}
\left\{\begin{aligned}
	  \dot{z}_1(t) &= L_1 z_1(t),  \qquad&t\ge0,\\
	  	  \dot{z}_2(t) &= L_2 z_2(t),  \qquad&t\ge0,\\
	  G_1 z_1(t) &=J^*K_2z_2(t), &t\ge0,\\
	  G_2z_2(t) &= -JK_1 z_1(t), &t\ge0,\\
	  z_1(0)&\in X_1,\ z_2(0)\in X_2.
	\end{aligned}\right.
\end{equation}
We may reformulate this coupled system as an abstract Cauchy problem
\begin{equation}\label{eq:ACP}
\left\{\begin{aligned}
	  \dot{z}(t) &= A z(t),  \qquad t\ge0,\\
	  	  z(0) &=z_0  ,
	\end{aligned}\right.
\end{equation}
for $z(\cdot)=(z_1(\cdot),z_2(\cdot))^\top$ on $X=X_1\times X_2$, where $z_0=(z_1(0),z_2(0))^\top\in X$  and $A\colon D(A)\subseteq X\to X$ is defined by $A=\diag(L_1,L_2)$ with domain 
$$D(A)=\bigg\{\begin{pmatrix}x_1\\x_2
\end{pmatrix}\in D(L_1)\times D(L_2): G_1x_1=J^*K_2x_2,\ G_2x_2=-JK_1x_1\bigg\}.$$
We denote by $A_j$ the restriction of $L_j$ to $\Ker G_j$ and we denote the transfer functions of the two boundary nodes by  $H_j$ and $P_j$  for $j=1,2$.

The following is our main result. It allows us to deduce a growth bound for the resolvent of $A$ from information about the boundary nodes $(G_1+J^*QJK_1,L_1, K_1)$ and $(G_2, L_2, K_2)$ where $Q\in \B(U_2)$ is a self-adjoint operator such that $Q\geq 0$. 
We note that $(G_1+J^* QJK_1,L_1, K_1)$ is indeed an impedance passive boundary node by Remark~\ref{rem:Qweakerconds}. One may think of it as a damped version of the boundary node $(G_1,L_1, K_1)$, and in the coupled system the damping is replaced by the interconnection with the boundary node $(G_2,L_2,  K_2)$.

\begin{thm}\label{thm:res}
Let $(G_1, L_1, K_1)$ and $(G_2,L_2,  K_2)$ be two impedance passive boundary nodes on the spaces $(U_1,X_1,U_1)$ and $(U_2,X_2,U_2)$, respectively, 
let $J\in\B(U_1,U_2)$ and let $Q\in \B(U_2)$ be a self-adjoint operator such that $Q\geq 0$. Let $A_0$ 
 denote the restriction of $L_1$ to $\Ker (G_1+J^*QJK_1)$
and let $H_0$ be the transfer function of $(G_1+J^\ast QJ,L_1,K_1)$.
 Suppose there exists a non-empty set $E\subseteq\{s\in\RR:is\in \rho(A_0)\cap\rho(A_2)\}$, and let 
$N_0,M_0,M_2\colon E\to[r,\infty)$, for some $r>0$, be such that 
$\norm{H_0(is)}\leq N_0(s)$, $s\in E$, and
$$\|(is-A_j)\inv\|\le M_j(s),\qquad s\in E,$$
for  $j=0,2$. Suppose furthermore that there exists a function $\eta\colon E\to(0,\infty)$ such that 
$\R P_2(is)\ge\eta(s)I$ for all $s\in E.$
Then $A$ generates a contraction semigroup on $X$, $iE\subseteq\rho(A)$ and
\begin{equation}\label{eq:res_est}
\|(is-A)\inv\|\lesssim
M_0(s)
+N_0(s)^2M_2(s)^2 \frac{\mu(s)}{\eta(s)},
\qquad s\in E,
\end{equation}
where $\mu(s)=1+\|P_2(1+is)\|^2$ for  $s\in \RR$.
\end{thm}

\begin{rem}\label{rem:bdd}
Theorem~\ref{thm:res} is particularly useful if we may choose $E=\RR$ or $E=\RR\setminus\{0\}$. 
We also note that the estimate in~\eqref{eq:res_est} simplifies in some important cases. For instance, if the semigroup generated by $A_j$ is exponentially stable for $j=0$ or $j=2$, then we may take the corresponding function $M_j$ to be constant. 
Moreover,
if $Q\geq cI$ for some $c>0$ (in particular, if $Q=I$), then Proposition~\ref{prp:Q} implies that
 $N_0(s)^2\lesssim M_0(s)$, $s\in E$, on the right-hand side of~\eqref{eq:res_est}.
Finally,
 by \cite[Lem.~13.1.10]{JacZwa12book} the function $\mu$ is bounded if $(G_2,L_2,K_2)$ is an \emph{(externally) well-posed} boundary node in the sense of \cite[Def.~13.1.3]{JacZwa12book}.
\end{rem}

\begin{rem}\label{rem:Riesz_res}
In coupled PDE models it is often natural to consider boundary nodes in which the codomains either of $G_1$ and $K_1$ or of $G_2$ and $K_2$ do not coincide but instead are each other's (conjugate) dual spaces with respect to an intermediate \emph{pivot space}. As described in Remark~\ref{rem:Riesz}, this case can be reduced to the situation considered in Theorem~\ref{thm:res}, and in particular the result still applies once the appropriate Riesz--Fréchet isomorphism is taken into consideration.
\end{rem}

The first step towards proving Theorem~\ref{thm:res} is the following  result.

\begin{lem}\label{lem:coupled}
Let $X_1, X_2$ and $U_1, U_2$ be Hilbert spaces, and let $(G_1, L_1, K_1)$ and $(G_2, L_2, K_2)$ be two impedance passive boundary nodes on $(U_1,X_1,U_1)$ and $(U_2,X_2,U_2)$, respectively, and let $J\in\B(U_1,U_2)$. Suppose that $\gl\in\rho(A_2)\cap\overline{\C_+}$ is such that $\R P_2(\gl)\ge cI$ for some $c>0$, and define $S_\gl\colon D(S_\gl)\subseteq X_1\to X_1$ by
$$S_\gl=\gl-L_1,\qquad D(S_\gl)=\Ker(G_1+J^*P_2(\gl)JK_1).$$
Then $\gl\in \rho(A)$ if and only if $S_\lambda$ is invertible, and if $S_\lambda$ is invertible then there exists $H_\gl\in\B(U_1,X)$ such that $\Ran H_\gl\subseteq \Ker(\gl-L_1)$, $$(G_1+J^*P_2(\gl) JK_1)H_\gl=I$$
and $(\gl-A)^{-1}$ is given by
$$\begin{pmatrix}
S_\gl\inv & H_\gl J^* K_2 (\gl-A_2)\inv\\
-H_2(\gl)JK_1 S_\gl\inv & (\gl-A_2)\inv -H_2(\gl)JK_1 H_\gl J^* K_2 (\gl-A_2)\inv 
\end{pmatrix}.$$
Finally, if $\gl\in i\RR$ then $S_\gl$ is invertible if and only if it is bounded below.
\end{lem}

\begin{proof}
Let $\gl\in\rho(A_2)\cap\overline{\C_+}$ be such that $\R P_2(\gl)\ge cI$ for some $c>0$. By Proposition~\ref{prp:Q},  $(G_1+J^*P_2(\gl)JK_1,L_1,K_1)$ is an impedance passive boundary node. Suppose first that $S_\gl$ is invertible. Then the existence of $H_\gl$ with the required properties follows from \cite[Rem.~10.1.5 \& Prop.~10.1.2]{TucWei09book}. We will show that $\gl\in\rho(A)$ and verify the formula for $(\gl-A)\inv$. Let
$$R_\gl=\begin{pmatrix}
S_\gl\inv & H_\gl J^*K_2 (\gl-A_2)\inv\\
-H_2(\gl)JK_1 S_\gl\inv & (I -H_2(\gl)JK_1 H_\gl J^*K_2) (\gl-A_2)\inv 
\end{pmatrix}.$$
We have $G_2H_2(\gl)=I$ and $K_2H_2(\gl)=P_2(\gl)$ by Proposition~\ref{prp:HPproperties}(a), and $G_2(\gl-A_2)\inv=0$ as $\Ran(\gl-A_2)\inv=D(A_2)=\Ker G_2$. Straightforward calculations based on these identities and the defining properties of $H_\gl$ show that $R_\gl$ maps into $D(A)$ and that $(\gl-A)R_\gl=I$. It remains to prove that $R_\gl$ is a left-inverse of $\gl-A$. To this end, let $x=(x_1,x_2)^\top\in D(A)$ and define $y=(y_1,y_2)^\top$ by $y=R_\gl(\gl-A)x$. We aim to show that $y=x$. Using the fact that $S_\gl\inv$ maps into $\Ker(G_1+J^*P_2(\gl)JK_1)$, the properties of $H_\gl$ and Proposition~\ref{prp:HPproperties}(a) we obtain
$$\begin{aligned}
(G_1+J^*P_2(\gl)JK_1)y_1&=J^*K_2(\gl-A_2)\inv(\gl-L_2)x_2\\&=J^*K_2x_2-J^*P_2(\gl)G_2x_2\\&=(G_1+J^*P_2(\gl)JK_1)x_1,
\end{aligned}$$
where in the final step we have used that $J^*K_2x_2=G_1x_1$ and $ G_2x_2=-JK_1x_1$. Hence $y_1-x_1\in \Ker(G_1+J^*P_2(\gl)JK_1)=D(S_\gl)$. Moreover, since $\Ran H_\gl\subseteq\Ker(\gl-L_1)$, we have
$$(\gl-L_1)y_1=(\gl-L_1)S_\gl\inv(\gl-L_1)x_1=(\gl-L_1)x_1,$$
and hence $y_1-x_1\in\Ker S_\gl=\{0\}$, giving $y_1=x_1$. Another application of Proposition~\ref{prp:HPproperties}(a) gives 
$$\begin{aligned}
y_2&=-H_2(\gl)JK_1  y_1+(\gl-A_2)\inv(\gl-L_2)x_2\\&=x_2-H_2(\gl)(JK_1  x_1+G_2x_2)=x_2,
\end{aligned}$$
and hence $y=x$. Thus $\gl\in\rho(A)$ and  $R_\gl=(\gl-A)\inv$, as required. 

Now suppose conversely that $\gl\in\rho(A)$. We prove that $S_\gl$ is invertible. In order to prove that it is injective, suppose that $x_1\in\Ker S_\gl$. Then $(G_1+J^*P_2(\gl)JK_1)x_1=0$ and $(\gl-L_1)x_1=0$. Let $x_2=-H_2(\gl)JK_1x_1\in D(L_2)$. Then $x_2\in \Ker(\gl-L_2)$, $G_2x_2=-JK_1x_1$ and $J^*K_2x_2=-J^*P_2(\gl)JK_1x_1=G_1x_1$ by Proposition~\ref{prp:HPproperties}(a), so $(x_1,x_2)^\top\in \Ker (\gl-A)$. Since $\gl-A$ is assumed to be invertible, this implies in particular that $x_1=0$, and hence $S_\gl$ is injective. Now suppose that $y_1\in X_1$. Let $y=(y_1,0)^\top\in X$ and define $x=(x_1,x_2)^\top\in D(A)$ by $x=(\gl-A)\inv y$. Since $x_2\in\Ker(\gl-L_2)$ we may use Proposition~\ref{prp:HPproperties}(a) to obtain
$$G_1x_1=J^*K_2x_2=J^*K_2H_2(\gl)G_2x_2=-J^*P_2(\gl)JK_1x_1,$$
 and hence $x_1\in\Ker(G_1+J^*P_2(\gl)JK_1)=D(S_\gl)$. Since $S_\gl x_1=(\gl-L_1)x_1=y_1$, it follows that $S_\gl$ is surjective. Since $L_1$ is closed by Lemma~\ref{lem:closed}, a straightforward argument shows that $S_\gl$ is closed. It follows that $S_\gl$ is invertible, as required.

Finally, since the restriction of $L_1$ to $\Ker(G_1+J^*P_2(\gl)JK_1)$ generates a contraction semigroup on $X_1$, it follows from \cite[Lem.~2.3]{AreBat88}, or alternatively from \cite[Prop.~4.3.1 \& Cor.~4.3.5]{AreBat11book}, that if $\gl\in i\RR$ then $S_\gl$ has dense range whenever it is injective. In particular, $S_\gl$ is invertible if and only if it is bounded below.
\end{proof}

\begin{prp}\label{prp:gen}
Let $(G_1, L_1, K_1)$ and $(G_2, L_2, K_2)$ be two impedance passive boundary nodes on the spaces $(U_1,X_1,U_1)$ and $(U_2,X_2,U_2)$, respectively, and let $J\in\B(U_1,U_2)$. If $\R P_2(\gl) $ is invertible for some $\gl\in\rho(A_2)\cap \overline{\C_+}$, then $A$ generates a contraction semigroup on $X$.
\end{prp}

\begin{proof}
Let $x=(x_1,x_2)^\top\in D(A).$ By impedance passivity of $(G_1, L_1, K_1)$ and $(G_2, L_2, K_2)$ we have
$$\begin{aligned}
\R\langle Ax,x\rangle_X&=\R\langle L_1x_1,x_1\rangle_{X_1}+\R\langle L_2x_2,x_2\rangle_{X_2}\\
&\le \R\langle G_1x_1,K_1x_1\rangle_{U_1}+\R\langle G_2x_2,K_2x_2\rangle_{U_2}\\
&= \R\langle J^*K_2x_2,K_1x_1\rangle_{U_1}-\R\langle JK_1x_1,K_2x_2\rangle_{U_2}=0,
\end{aligned}$$
so $A$ is dissipative. Let $\gl\in\overline{\C_+}\cap\rho(A_2)$ be such that $\R P_2(\gl) $ is invertible. By continuity of the map $\gl\mapsto \R P_2(\gl) $ on a neighbourhood of $\gl$ we may assume without loss of generality that $\gl\in\C_+$. Then  $\gl\in\rho(A_1)$ and $\R P_j(\gl)\ge0$ for $j=1,2$, by Proposition~\ref{prp:HPproperties}(b). Since $\R P_2(\gl)$ is invertible, it follows that $\R P_2(\gl)\ge cI$ for some $c>0$. Using \cite[Lem.~A.1(a)]{Pau19}, we see that $P_2(\gl)$ is invertible and $\R P_2(\gl)\inv\ge c\|P_2(\gl)\|^{-2}I$. Hence
$$\R(J^*P_1(\gl)J+P_2(\gl)\inv)\ge\R P_2(\gl)\inv\ge c\|P_2(\gl)\|^{-2}I,$$
so $J^*P_1(\gl)J+P_2(\gl)\inv$ is invertible. Define $R_\gl \in\B(X_1)$ by 
$$R_\gl=(\gl-A_1)\inv-H_1(\gl)J^*J(J^*P_1(\gl)J+P_2(\gl)\inv)\inv J^*JK_1(\gl-A_1)\inv.$$
Straightforward computations using Proposition~\ref{prp:HPproperties}(a) show that $\Ran R_\gl\subseteq \Ker(G_1+J^*P_2(\gl)JK_1)=D(S_\gl)$ and that $R_\gl$ is the inverse for $S_\gl$. Hence $S_\gl$ is invertible, and it follows from Lemma~\ref{lem:coupled} that $\gl\in\rho(A)$. In particular, $\gl-A$ is surjective and hence maximally dissipative. By the Lumer--Phillips theorem $A$ generates a contraction semigroup.
\end{proof}

We  now  prove  Theorem~\ref{thm:res}.

\begin{proof}[Proof of Theorem \textup{\ref{thm:res}}]
Let $s\in E$. Since ${is}\in\rho(A_2)$ and $\R P_2({is})$ is invertible by our assumptions, it follows from Proposition~\ref{prp:gen} that $A$ generates a contraction semigroup on $X$. Furthermore, by Lemma~\ref{lem:coupled} we have ${is}\in\rho(A)$ provided the operator $S_{is}\colon D(S_{is})\subseteq X_1\to X_1$ defined by
$S_{is}={is}-L_1$ with domain $D(S_{is})=\Ker(G_1+J^*P_2({is})JK_1)$ is bounded below. Let $x\in D(S_{is})$ and let $y=S_{is} x$. Then $(G_1+J^* P_2(is)JK_1)x=0$, so 
$(G_1+J^*QJK_1)x=J^*(Q-P_2(is))JK_1x$.
 Moreover $(is-L_1)x=y$, so $(is-A_0)\inv y=x-H_0(is)(G_1+J^*QJK_1)x$ by Proposition~\ref{prp:HPproperties}(a), where $H_0$ denotes the transfer function associated with the boundary node $(G_1+J^*QJK_1,L_1,K_1)$. Hence
\begin{equation}\label{eq:x}
x=({is}-A_0)\inv y+H_0({is})J^*(Q-P_2({is}))JK_1x.
\end{equation}
 Applying Proposition~\ref{prp:Q}(a) 
with $P_2(is)$ taking the place of the operator $Q$ appearing there,
 we find that
\begin{equation}\label{eq:real_part}
\R\langle y,x\rangle=-\R\langle L_1x,x\rangle\ge\eta(s)\|JK_1x\|^2,
\end{equation}
and hence $\|JK_1x\|^2\le\eta(s)\inv\|x\|\|y\|$. 
Using 
the formula~\eqref{eq:x}, the estimates  $\|(is-A_0)\inv\|\le M_0(s)$ and $\norm{H_0(is)}\leq N_0(s)$  along with Young's inequality we get
$$\begin{aligned}
\|x\|&\le M_0(s)\|y\|+N_0(s)\|J\|(\norm{Q}+\|P_2({is})\|)\left(\frac{\|x\|\|y\|}{\eta(s)}\right)^{1/2}\\
&\leq \frac{\norm{x}}{2} + \left(M_0(s) + \frac{N_0(s)^2\|J\|^2(\norm{Q}^2+\norm{P_2({is})}^2)}{\eta(s)}  \right)\norm{y}.
\end{aligned}$$
Since $\|P_2({is})\|\ge\eta(s)$ and, by Proposition~\ref{prp:HPproperties}(b), 
$\|P_2({is})\|\lesssim  M_2(s)\|P_2(1+is)\|,$
we deduce that 
$\|x\|\lesssim (M_0(s) + N_0(s)^2M_2(s)^2 \mu(s)\eta(s)\inv)\|S_{is} x\|,$
where the implicit constant is independent of $s$ (as will be the case throughout the remainder of this proof). In particular, the operator $S_{is}$ is bounded below and hence invertible, so that ${is}\in\rho(A)$. Furthermore, 
\begin{equation}\label{eq:TL}
\|S_{is}\inv\|\lesssim M_0(s) + N_0(s)^2M_2(s)^2 \frac{\mu(s)}{\eta(s)},
\end{equation}
 giving an upper bound for the top left entry in the matrix representing $({is}-A)\inv$ in Lemma~\ref{lem:coupled}. 

We now estimate the remaining three entries of $(is-A)\inv$.  Since $s\in E$ is fixed and $\R P_2(is)\geq \eta(s)I$ with $\eta(s)>0$, we have from Proposition~\ref{prp:Q} that 
$(G_1+J^\ast P_2(is)JK_1,L_1,K_1)$
 is an impedance passive boundary node. 
We note that $H_{is}$ coincides with the first transfer function associated with this boundary node evaluated at $is$.
Since $S_{is} $ is the restriction of $is-L_1$ to $\Ker (G_1+J^\ast P_2(is) JK_1)$ and since $S_{is}$ is invertible, the estimates in Proposition~\ref{prp:Q}(c) give
$$
\begin{aligned}
\norm{J K_1 S_{is}\inv }^2 &\leq \eta(s)\inv \norm{S_{is}\inv},\\
\norm{ H_{is} J^\ast}^2 &\leq \eta(s)\inv \norm{S_{is}\inv},\\
\norm{ JK_1H_{is} J^\ast} &\leq \eta(s)\inv .
\end{aligned}
$$
Since Proposition~\ref{prp:HPproperties}(b) implies that $\|H_2({is})\|\lesssim M_2(s)\|P_2(1+is)\|^{1/2}$ and $\|K_2(is-A_2)\inv\|\lesssim M_2(s)\|P_2(1+is)\|^{1/2}$, we may estimate the top right and bottom left entries of $(is-A)\inv $ by
\begin{equation*}\label{eq:TR}
\begin{aligned}
\| H_{is}J^*K_2 (is-A_2)\inv\|
&\lesssim
\frac{M_2(s)\mu(s)^{1/4}}{\eta(s)^{1/2}} \left( M_0(s) + N_0(s)^2M_2(s)^2 \frac{\mu(s)}{\eta(s)} \right)^{1/2}
\\
&\lesssim
\frac{M_2(s)\mu(s)^{1/4}}{\eta(s)^{1/2}}  M_0(s)^{1/2} + N_0(s)M_2(s)^2 \frac{\mu(s)^{3/4}}{\eta(s)} 
\\
&\lesssim
M_0(s) 
 + N_0(s)^2M_2(s)^2 \frac{\mu(s)}{\eta(s)} 
\end{aligned}
\end{equation*}
and, analogously,
\begin{equation*}\label{eq:BL}
\|{H_2({is})JK_1S_{is}\inv}\|\lesssim
M_0(s) 
 + N_0(s)^2M_2(s)^2 \frac{\mu(s)}{\eta(s)}.
\end{equation*}
Finally, since
$$M_2(s)\le M_2(s)\frac{\|P_2(is)\|}{\eta(s)}\lesssim
 M_2(s)^2\frac{\|P_2(1+is)\|}{\eta(s)},$$
 we may bound the bottom right entry of $(is-A)\inv$ by
\begin{equation*}\label{eq:BR}
 M_2(s)+\|H_2(is)\|\|JK_1 H_{is}J^*\|\| K_2 (is-A_2)\inv\|\lesssim M_2(s)^2\frac{\|P_2(1+is)\|}{\eta(s)}.
\end{equation*}
Now \eqref{eq:res_est} follows from the formula for $(is-A)\inv$ given in Lemma~\ref{lem:coupled} together with~\eqref{eq:TL} and our estimates for the remaining three matrix entries.

\end{proof}

\begin{rem}\label{rem:spec}
A variant of Theorem~\ref{thm:res} remains true even without the real part of  $P_2(is)$ being strictly positive for certain $is\in\rho(A_1)\cap\rho(A_2)$. 
This can be seen by using a perturbation formula similar to \eqref{eq:AQ_inv} to analyse $S_{is}\inv$.
Indeed, given $is\in\rho(A_1)\cap\rho(A_2)$ it is possible to show that $is\in\rho(A)$ if and only if $I+J^*P_2(is)JP_1(is)$ is invertible,
and in this case we may obtain a bound for
 $\|(is-A)\inv\|$ in terms of
 $M_2(s)$,
 $\mu(s)$ and the norms of
$(is-A_1)\inv$, $P_1(1+is)$ and  $(I+J^*P_2(is)JP_1(is))\inv$.
\end{rem}

Since Theorem~\ref{thm:res} applies  in cases where $\sigma(A)\cap i\RR$ may be non-empty,  it is sometimes useful to know that $\sigma(A)\cap i\RR$ cannot be too large. This is true in particular whenever $A$ has compact resolvent. The following simple lemma gives natural sufficient conditions for this to be the case.

\begin{lem}
\label{lem:coupledCompact}
Let $(G_1, L_1, K_1)$ and $(G_2, L_2, K_2)$ be two impedance passive boundary nodes on  $(U_1,X_1,U_1)$ and $(U_2,X_2,U_2)$, respectively, and assume that $\R P_2(\gl) $ is invertible for some $\gl\in\rho(A_2)\cap \overline{\C_+}$. If both $A_1 $ and $A_2$ have compact resolvent and if 
$J\in \B(U_1,U_2)$ is compact,
then $A$ has compact resolvent.
\end{lem}

\begin{proof}
As in the proof of Proposition~\ref{prp:gen} we may assume without loss of generality  that $\R P_2(\gl)$ is invertible for some $\gl\in\C_+$.
The proof of Proposition~\ref{prp:gen} shows that $S_\gl$ is invertible and that
$$S_\gl\inv = (\gl-A_1)\inv-H_1(\gl)J^*J(J^*P_1(\gl)J+P_2(\gl)\inv)\inv J^*JK_1(\gl-A_1)\inv.$$
 Hence $\gl\in\rho(A)$ and $(\gl-A)\inv$ has the form given in  
Lemma~\ref{lem:coupled}. Since $(\gl-A_1)\inv$ and $J$ are assumed to be compact, the same is true for
 $S_\gl\inv$. Our assumptions further ensure that the remaining operators appearing in the formula in~Lemma~\ref{lem:coupled} are compact,
 so
$(\gl-A)\inv$ too is compact.
\end{proof}

We close this section by analysing a boundary control system coupled with an infinite-dimensional linear system. More specifically, we consider the system
$$
\left\{\begin{aligned}
	  \dot{z}_1(t) &= L_1 z_1(t),  \qquad&t\ge0,\\
	  	  \dot{z}_2(t) &= A_2 z_2(t)-B_2JK_1z_1(t),  \qquad&t\ge0,\\
	  G_1 z_1(t) &=J^*C_2z_2(t)-J^*D_2JK_1z_1(t), &t\ge0,\\
	  z_1(0)&\in X_1,\ z_2(0)\in X_2.
	\end{aligned}\right.
$$
Here $X_1,X_2, U_1$ and $U_2$ are Hilbert spaces, $(G_1,L_1,K_1)$ is an impedance passive boundary node on $(U_1,X_1,U_1)$, $A_2\colon D(A_2)\subseteq X_2\to X_2$ is the generator of a contraction semigroup on $X_2$ and $B_2\in\B(U_2,X_2)$, $C_2\in\B(X_2,U_2)$, $D_2\in\B(U_2)$ and $J\in\B(U_1,U_2)$ are bounded operators. In applications, the linear dynamical system $(A_2,B_2,C_2,D_2)$ might describe an ordinary differential equation, or a PDE with \emph{distributed} input and output. We reformulate this coupled system as an abstract Cauchy problem
\begin{equation*}\label{eq:ACP_2}
\left\{\begin{aligned}
	  \dot{z}(t) &= A z(t),  \qquad t\ge0,\\
	  	  z(0) &= z_0,
	\end{aligned}\right.
\end{equation*}
where $z_0=(z_1(0),z_2(0))^\top\in X$, $X$ is the product Hilbert space $X_1\times X_2$ and $A\colon D(A)\subseteq X\to X$ is the operator
$$A=
\begin{pmatrix}
L_1&0\\-B_2 JK_1&A_2
\end{pmatrix}$$
with domain 
$$D(A)=\bigg\{\begin{pmatrix}x_1\\x_2 
\end{pmatrix}\in D(L_1)\times D(A_2): (G_1+J^*D_2JK_1)x_1=J^*C_2x_2\bigg\}.$$

\begin{thm}\label{thm:coupled}
Given Hilbert spaces $X_1, X_2, U_1$ and $U_1$ let $(G_1,L_1,K_1)$ be an impedance passive boundary node on $(U_1,X_1,U_1)$, let $A_2\colon D(A_2)\subseteq X_2\to X_2$ be a  linear operator such that $\gl-A_2$ is surjective for some $\gl\in\C_+$ and let $B_2\in\B(U_2,X_2)$, $C_2\in\B(X_2,U_2)$ and $D_2\in\B(U_2)$ be such that 
\begin{equation}\label{eq:imped_pass}
\R\<A_2x+B_2u,x\>_{X_2}\le\R\<C_2x+D_2u,u\>_{U_2},\qquad x\in D(A_2),\ u\in U_2.
\end{equation}
Let $J\in\B(U_1,U_2)$, let $Q\in \B(U_2)$ be a self-adjoint operator such that $Q\geq 0$,
 let
$A_0$ denote the restriction of $L_1$ to $\Ker (G_1+J^*QJK_1)$
and let $H_0$ be the transfer function of $(G_1+J^\ast QJ,L_1,K_1)$.
Suppose there exists a non-empty set $E\subseteq\{s\in\RR:is\in \rho(A_0)\cap\rho(A_2)\}$. Furthermore, let
$N_0,M_0,M_2\colon E\to[r,\infty)$, for some $r>0$, be such that 
$\norm{H_0(is)}\leq N_0(s)$, $s\in E$, and
$$\|(is-A_j)\inv\|\le M_j(s),\qquad s\in E,$$
for  $j=0,2$, and suppose there exists a function $\eta\colon E\to(0,\infty)$ such that 
\begin{equation}\label{eq:R_lb}
\R (C_2(is-A_2)\inv B_2+D_2)\ge\eta(s)I,\qquad s\in E.
\end{equation}
Then $A$ generates a contraction semigroup on $X$, $iE\subseteq\rho(A)$ and
\begin{equation}\label{eq:res_est2}
\|(is-A)\inv\|\lesssim
M_0(s)
+ \frac{N_0(s)^2M_2(s)^2}{\eta(s)},
\qquad s\in E.
\end{equation}
\end{thm}

\begin{proof}
We begin by noting that \eqref{eq:imped_pass} applied with $u=0$ shows that $A_2$ is dissipative. Since $\gl-A_2$ is assumed to be surjective for some $\gl\in\C_+$ it follows from the Lumer--Phillips theorem that $A_2$ generates a contraction semigroup on $X_2$.
Let $x=(x_1,x_2)^\top\in D(A)$. Then $G_1x_1=J^*C_2x_2-J^*D_2JK_1x_1$ and hence
$$\begin{aligned}
\R\<Ax,x\>_X&=\R\<L_1x_1,x_1\>_{X_1}+\R\<A_2x_2-B_2JK_1x_1,x_2\>_{X_2}\\
&\le \R\<G_1x_1,K_1x_1\>_{U_1}+\R \<C_2x_2 - D_2JK_1x_1, -JK_1x_1\>_{U_2}=0.
\end{aligned}$$
It follows that $A$ is dissipative. Define $P_2\colon\rho(A_2)\to\B(U_2)$ and $H_2\colon\rho(A_2)\to\B(U_2,X_2)$ by $P_2(\gl)=C_2(\gl-A_2)\inv B_2+D_2$ and $H_2(\gl)=(\gl-A_2)\inv B_2$, respectively, for all $\gl\in\rho(A_2)$. Moreover,  define $S_\gl\colon D(S_\gl)\subseteq X_1\to X_1$ by 
$$S_\gl=\gl-L_1,\qquad D(S_\gl)=\Ker(G_1+J^*P_2(\gl)JK_1)$$
for all $\gl\in\rho(A_2)$.  As in the proof of Lemma~\ref{lem:coupled} we may show that if $\gl\in\rho(A_2)\cap\overline{\C_+}$ is such that 
$\R P_2(\gl)\ge cI$ for some $c>0$ and 
$S_\gl$ is invertible, then $\gl\in\rho(A)$. Moreover, there exists $H_\gl\in\B(U_1,X_1)$ such that $\Ran H_\gl\subseteq \Ker(\gl-L_1)$, $(G_1+J^*P_2(\gl) JK_1)H_\gl=I$
and the inverse of $\gl-A$ is given by 
$$\begin{pmatrix}
S_\gl\inv & H_\gl J^*C_2 (\gl-A_2)\inv\\
-H_2(\gl)JK_1 S_\gl\inv & (\gl-A_2)\inv -H_2(\gl)JK_1 H_\gl J^*C_2 (\gl-A_2)\inv 
\end{pmatrix}.$$
Finally, if $\gl\in i \RR$ then $S_\gl$ is invertible provided it is bounded below. We note that $\|P_2(is)\|\lesssim M_2(s)$ for $s\in E$ and, since $A_2$ generates a contraction semigroup on $X_2$, we have $\sup_{s\in\RR}\|(1+is-A_2)\inv\|<\infty$. Now let $s\in E$, and let $\gl=is$. Then $\gl\in\rho(A_2)\cap\overline{\C_+}$ by definition of $E$ and  $\R P_2(\gl)\geq \eta(s) I$ by~\eqref{eq:R_lb}. Since the operators $B_2$, $C_2$ and $D_2$ are bounded, we may show as in the proof of Theorem~\ref{thm:res} that $S_\gl$ is bounded below and therefore invertible. Thus $\gl\in\rho(A)$, so $A$ is maximally dissipative and hence generates a contraction semigroup on $X$ by the Lumer--Phillips theorem. 
The estimate $ \|(is-A)\inv\| \lesssim M_0(s) + N_0(s)^2 M_2(s)^2\eta(s)\inv $ for $s\in E$ follows more or less exactly as in the proof of Theorem~\ref{thm:res}, using the fact that the structure of $(\gl-A)\inv$ is analogous to that of the matrix appearing in Lemma~\ref{lem:coupled}.
\end{proof}

\begin{rem}
 In the terminology of~\cite[Def.~4.1]{Sta02}, condition~\eqref{eq:imped_pass} means that the  system~$(A_2,B_2,C_2,D_2)$ is \emph{impedance passive}. It follows in particular that $\R (C_2(\gl-A_2)^{-1}B_2+D_2)\ge0$ for all $\gl\in\rho(A_2) \cap\overline{\C_+}$.
One important case in which condition~\eqref{eq:imped_pass} is satisfied is if $A_2$ is assumed to be the generator of a contraction semigroup on $X_2$, $B_2=C_2^*$ and $D_2=0$ (or, more generally, $\R D_2\ge0$).
\end{rem}

\begin{rem}\phantomsection\label{rem:dyn}
The simplifications described in Remark~\ref{rem:bdd} also apply in the context of Theorem~\ref{thm:coupled}, and in particular 
we have 
 $N_0(s)^2\lesssim M_0(s)$, $s\in E$,
provided that $Q\geq cI$ for some $c>0$ (in particular, if $Q=I$).
Furthermore,  Remark~\ref{rem:spec} applies also in the setting of Theorem~\ref{thm:coupled}. In particular, given $is\in\rho(A_1)\cap\rho(A_2)$ it is possible to show that $is\in\rho(A)$ if and only if $I+J^*P_2(is)JP_1(is)$ is invertible, where $P_2(is)=C_2(is-A_2)\inv B_2+D_2$.
\end{rem}

\begin{rem}\phantomsection\label{rem:Riesz_dyn}
We point out that, as in Remark~\ref{rem:Riesz_res}, Theorem~\ref{thm:coupled} extends to the case where the codomains of $G_1$ and $K_1$ are not the same but merely each other's (conjugate) dual spaces, as described in Remark~\ref{rem:Riesz}.
\end{rem}

\section{Stability of coupled PDE models}\label{sec:appl}

In this section we shall apply the results of Section~\ref{sec:coupled} in order to study the quantitative asymptotic behaviour of several concrete examples of coupled PDE models.  We approach such systems by first reformulating them as an abstract Cauchy problem of the form
\begin{equation}\label{eq:ACP2}
\left\{\begin{aligned}
	  \dot{z}(t) &= A z(t),  \qquad t\ge0,\\
	  	  z(0) &=z_0 \in X,
	\end{aligned}\right.
\end{equation}
where $A\colon D(A)\subseteq X\to X$ generates a contraction semigroup $\T$ on some appropriately chosen Hilbert space $X$. In this case the orbit $T(\cdot)z_0$ corresponding to the initial data $z_0\in X$ is a  \emph{(mild) solution} of the abstract Cauchy problem~\eqref{eq:ACP2}, and it is a \emph{classical solution} of~\eqref{eq:ACP2} precisely when $z_0\in D(A)$. We refer the reader to~\cite[Ch.~3]{AreBat11book} for further details on abstract Cauchy problems and $C_0$-semigroups.

A $C_0$-semigroup $\T$ on a Banach space $X$ is said to be \emph{strongly stable} if $\|T(t)x\|\to0$ as $t\to\infty$ for all $x\in X$, and the semigroup $\T$ is said to be \emph{uniformly exponentially stable} if there exist constants $M,\omega>0$ such that $\|T(t)\|\le Me^{-\omega t}$ for all $t\ge0$.  Theorems~\ref{thm:res} and \ref{thm:coupled} may be used to give sufficient conditions for classes of the abstract coupled systems we have studied to lead to an exponentially stable semigroup and hence give a uniform decay rate for all solutions; see in particular Remark~\ref{rem:bdd}. However, in many cases of interest the resolvent is not uniformly bounded along the imaginary axis, and in this case one cannot hope to obtain a decay rate that is uniform for \emph{all}  solutions. Nevertheless, it is possible to obtain rates of decay for all \emph{classical} solutions of the abstract Cauchy problem provided one is able to estimate the growth of the resolvent of the imaginary axis. We now formulate a general result which will allow us to convert the resolvent estimates obtained in Theorems~\ref{thm:res} and \ref{thm:coupled} into decay rates for suitable semigroup orbits. Recall to this end that a measurable function $M\colon \RR_+\to(0,\infty)$ is said to have \emph{positive increase} if there exist constants $c\in(0,1]$ and $\alpha,s_0>0$  such that 
$$\frac{M(\gl s)}{M(s)}\ge c\gl^\alpha,\qquad \gl\ge1,\ s\ge s_0.$$
This class of functions includes all \emph{regularly varying} functions, and in particular functions given by  $M(s)=Cs^\alpha\log(s)^\beta$ for $s\ge s_0$, where $C,\alpha>0,$  $s_0>1$  and $\beta\in\RR$; see \cite[Sect.~2]{BaChTo16} and \cite[Sect.~2]{RoSeSt19} for further details. Given a continuous non-increasing function $M\colon\RR_+\to(0,\infty)$ we denote by $M\inv$ the maximal right-inverse of $M$, defined by $M\inv(t)=\sup\{s\ge0:M(s)\le t\}$ for $t\ge M(0)$.

\begin{thm}\label{thm:sg}
Let $\T$ be a bounded $C_0$-semigroup on a Hilbert space $X$, with generator $A$, and suppose that $i\RR\subseteq\rho(A)$. Then the following hold:
\begin{enumerate}
\item[\textup{(a)}] The semigroup $\T$ is strongly stable.
\item[\textup{(b)}]  The semigroup $\T$ is uniformly exponentially stable if and only if $\sup_{s\in\RR}\|(is-A)\inv\|<\infty$.
\item[\textup{(c)}] Suppose that $M\colon[0,\infty)\to(0,\infty)$ is a continuous function of positive increase such that $\|(is-A)\inv\|\lesssim M(|s|)$ for all $s\in\RR$. Then 
$$\|T(t)z_0\|=O\bigg(\frac{1}{M\inv(t)}\bigg),\qquad t\to\infty,$$
for all $z_0\in D(A)$. Furthermore,  if there exists $\alpha>0$ such that $\|(is-A)\inv\|\lesssim 1+|s|^\alpha$ for all $s\in\RR$ then $\|T(t)z_0\|=o(t^{-1/\alpha})$ as $t\to\infty$ for all $z_0\in D(A)$.
\end{enumerate}
\end{thm}

\begin{proof}
Part~(a) is a special case of the theorem of Arendt--Batty and Lyubich--V\~{u} \cite[Thm.~V.2.21]{EngNag00book}, while part~(b) follows from the Gearhart--Prüss theorem \cite[Thm.~V.1.11]{EngNag00book}. For~(c), we note that $\|T(t)A\inv\|=O(M\inv(t)\inv)$ as $t\to\infty$ by~\cite[Prop.~2.2 \& Thm.~3.2]{RoSeSt19}, so the first statement follows from the fact that   $A\inv$ maps onto $D(A)$. The final statement is proved in~\cite[Thm.~2.4]{BorTom10}.
\end{proof}

\begin{rem}\label{rem:asymp}
In part~(c) it is also possible to pass, conversely, from a decay rate for classical solutions to a growth estimate for the resolvent of the semigroup generator along the imaginary axis. Indeed, suppose that $A$ generates a bounded $C_0$-semigroup on a Banach space $X$ and that $m\colon \RR_+\to(0,\infty)$ is a continuous non-increasing function such that $m(t)\to0$ as $t\to\infty$ and $\|T(t)(\gl-A)\inv\|\le m(t)$ for some $\gl\in\rho(A)$ and all $t\ge0$. Then $ i\RR\subseteq\rho(A)$ and
$$\|(is-A)\inv\|=O\bigg(m\inv\bigg(\frac{c}{|s|}\bigg)\bigg),\qquad |s|\to\infty,$$
for every $c\in(0,1)$. Here we take $m\inv$ to be the minimal right-inverse of $m$, defined by $m\inv(s)=\inf\{t\ge0:m(t)\ge s\}$ for $0< s\le m(0)$.
For a proof of this result we refer the reader to \cite[Prop.~1.3]{BatDuy08} and~\cite[Thm.~4.4.14]{AreBat11book}. A recent  survey of results in the asymptotic theory of $C_0$-semigroups more generally may be found in~\cite{ChSeTo20}.
\end{rem}

Combining Theorem~\ref{thm:sg} with Theorems~\ref{thm:res} and~\ref{thm:coupled} we obtain the following result describing the asymptotic behaviour of abstract coupled systems considered in Section~\ref{sec:coupled}. 

\begin{cor}\label{cor:decay}
Consider the setting of either Theorem~\textup{\ref{thm:res}}  or Theorem~\textup{\ref{thm:coupled}} with $E=\RR$, and let $M\colon[0,\infty)\to(0,\infty)$ be such that  
$$
M_0(s)+N_0(s)^2M_2(s)^2\frac{\mu(s)}{\eta(s)}\lesssim M(|s|),\qquad s\in\RR,
$$
where we set $\mu(s)=1$ for all $s\in\RR$ in the setting of Theorem~\textup{\ref{thm:coupled}}. Then the operator $A$ generates a contraction semigroup $\T$, and the following hold:
\begin{enumerate}
\item[\textup{(a)}]  The semigroup  $\T$  is strongly stable.
\item[\textup{(b)}] If $M$ may be taken to be bounded,
 then  $\T$ is uniformly exponentially stable. 
\item[\textup{(c)}]If $M$ is continuous and has positive increase, then
$$\|T(t)z_0\|=O\bigg(\frac{1}{M\inv(t)}\bigg),\qquad t\to\infty,$$
for all $z_0\in D(A)$. Furthermore,  if there exists $\alpha>0$  and $s_0\ge0$ such that $M(s)=1+s^\alpha$ for all $s\ge s_0$ then $\|T(t)z_0\|=o(t^{-1/\alpha})$ as $t\to\infty$ for all $z_0\in D(A)$.
\end{enumerate}
\end{cor}

 In the remainder of this section, we use our abstract results in order to study energy decay in several concrete PDE models. We reformulate each of these systems as an abstract Cauchy problem~\eqref{eq:ACP2} in such a way that the energy $E(t)$ of a solution with initial data $z_0\in X$, when defined in a natural way, is proportional to $\|T(t)z_0\|^2$ for all $t\ge0$.  

\subsection{A one-dimensional wave-heat system}\label{sec:wave_heat} We begin by considering the coupled wave-heat system on a one-dimensional spatial domain, namely 
\begin{equation}\label{eq:wave-heat}
\left\{
\begin{aligned}
y_{tt}(x,t)&=y_{xx}(x,t),\qquad & x\in(-1,0),\ t>0,\\
w_t(x,t)&=w_{xx}(x,t),& x\in(0,1),\ t>0,\\
y_x(-1,0)&=0,\quad w(1,t)=0 &t>0,\\
y_t(0,t)&=w(0,t),\quad y_x(0,t)=w_x(0,t), &t>0,
\end{aligned}
\right.
\end{equation}
subject to suitable initial conditions. Let $X_1=L^2(-1,0)\times L^2(-1,0)$ and $X_2=L^2(0,1)$, both endowed with their natural Hilbert space norms. Moreover, let $z_1(t)=(u(\cdot,t),v(\cdot,t))^\top$ for $t\ge0$, where $u=y_x$ and $v=y_t$, and let $z_2(t)=w(\cdot,t)$ for $t\ge0$. Given initial data $z_0=(z_1(0),z_2(0))^\top\in X=X_1\times X_2$, we wish to study the asymptotic behaviour as $t\to\infty$ of the total \emph{energy} $E(t)=E_{W}(t)+E_{H}(t)$ associated with the solution of the wave-heat system, where
$$E_{W}(t)=\frac12\int_{-1}^0 |y_x(x,t)|^2+|y_t(x,t)|^2\,\dd x,\qquad t\ge0,$$ 
and
$$E_{H}(t)=\frac12\int_0^1 |w(x,t)|^2\,\dd x,\qquad t\ge0,$$ 
are the energies associated with the wave part and the heat part, respectively. 
In order to apply our abstract results we will formulate the wave-heat system as an abstract Cauchy problem in which the norm of an orbit corresponds to the energy of the associated solution of~\eqref{eq:wave-heat}. 
We describe the generator of this abstract Cauchy problem using two impedance passive boundary nodes $(G_1,L_1,K_1)$ and $(G_2,L_2,K_2)$ for the wave part and the heat part, respectively, which are interconnected in a way that captures the boundary coupling in~\eqref{eq:wave-heat}.

For $j=1,2$, we define $L_j\colon D(L_j)\subseteq X_j\to X_j$ and $G_j, K_j\colon D(L_j)\subseteq X_j\to\C$ by 
$$L_1\begin{pmatrix}u\\v
\end{pmatrix}
=
\begin{pmatrix}v'\\u'
\end{pmatrix},
\qquad 
G_1\begin{pmatrix}u\\v
\end{pmatrix}
=v(0),\qquad 
K_1\begin{pmatrix}u\\v
\end{pmatrix}
=u(0)$$
for all $(u,v)^\top\in D(L_1)=\{(u,v)^\top\in H^1(-1,0)\times H^1(-1,0):u(-1)=0\},$ and 
$$L_2w=w''
\qquad 
G_2w
=-w'(0),\qquad 
K_2w=w(0)$$
for all $w\in D(L_2)=\{w\in H^2(0,1):w(1)=0\}.$ Then our wave-heat system~\eqref{eq:wave-heat} becomes an instance of the abstract coupled system~\eqref{eq:coupled_sys}. We verify that $(G_j,L_j,K_j)$ is an impedance passive boundary node on $(\C,X_j,\C)$ for $j=1,2$. If $x=(u,v)^\top\in D(L_1)$ then 
\begin{equation*}\label{eq:wh_ibp}
\R\<L_1x,x\>_{X_1}=\R\big(\<v',u\>_{L^2}+\<u',v\>_{L^2}\big)=\R u(0)\overline{v(0)}=\R\<G_1x,K_1x\>_{\C},
\end{equation*}
using integration by parts and the fact that $u(-1)=0$. It follows that the restriction $A_1$ of $L_1$ to $\Ker G_1$ is dissipative. Furthermore, by the Rellich--Kondrachov theorem the domain $D(A_1)=\Ker G_1$ of $A_1$ is compactly embedded in $X_1$, so $A_1$ has compact resolvent and in particular must be maximally dissipative, so $A_1$ generates a contraction semigroup by the Lumer--Phillips theorem. The operator $G_1$ is non-zero and therefore maps onto its codomain $\C$. Hence there exists $x_0\in D(L_1)$ such that $G_1x_0=1$, and we may define a right-inverse $G_1^r\in\B(\C,D(L_1))$ of $G_1$ by $G_1^r\gl=\gl x_0$ for all $\gl\in\C$. It follows that $(G_1,L_1,K_1)$ is a boundary node on $(\C,X_1,\C)$, and by our previous calculation using integration by parts it is impedance passive. Next we observe that the restriction $A_2$ of $L_2$ to $\Ker G_2$ is a negative self-adjoint operator and in particular generates a contraction semigroup by the Lumer--Phillips theorem. As before, the operator $G_2$ is non-zero and hence maps onto its codomain $\C$, so it has a right-inverse $G_2^r\in\B(\C,D(L_2))$. It follows that $(G_2,L_2,K_2)$ is a boundary node on $(\C,X_2,\C)$. Finally, if $w\in D(L_2)$ then using integration by parts and the fact that $w(1)=0$ we find that
$$\R\<L_2w,w\>_{X_2}=-\R w'(0)\overline{w(0)}-\|w\|_{L^2}^2\le\R\<G_2w,K_2w\>_{\C},$$
so the boundary node is impedance passive, as required. 

In order to apply Theorem~\ref{thm:res} with $J=1$ and $Q=1$ we first note that the semigroup on $X_1$ generated by the restriction $A_0$ of $L_1$ to $\Ker(G_1+K_1)$ corresponds to the wave equation on the interval $(-1,0)$ with a non-reflective boundary condition at $x=0$. This semigroup is nilpotent and in particular uniformly exponentially stable, so we have $i\RR\subseteq\rho(A_0)$ and $\sup_{s\in\RR}\|(is-A_0)\inv\|<\infty$. It follows that we may take the function $M_0$ to be constant. Since $A_2$ is negative and invertible, we have $i\RR\subseteq\rho(A_2)$ and $\sup_{s\in\RR}\|(is-A_2)\inv\|<\infty$. Hence we may also take $M_2$ to be constant. The transfer function $P_2$ of the boundary node $(G_2, L_2, K_2)$ satisfies $P_2(\gl)=w(0)$ for $\gl\in\overline{\C_+}$, where $w$ is the solution of the problem $\gl w- w''=0$ on $(0,1)$ subject to the boundary conditions $-w'(0)=1$ and $w(1)=0$. Thus
$$P_2(\gl)=\frac{\tanh\sqrt{\gl}}{\sqrt{\gl}},\qquad \gl\in\overline{\C_+}\setminus\{0\},$$
and $P_2(0)=1$. Here we take  the complex logarithm to have a branch cut along the negative real axis, so that the arguments of complex numbers in $\C\setminus(-\infty,0]$ lie in the interval $(-\pi,\pi)$. A simple  calculation shows that
$$\R P_2(is)=\frac{\sin \sqrt{2|s|}+\sinh \sqrt{2|s|}}{\sqrt{2|s|}\big(\cos \sqrt{2|s|}+\cosh \sqrt{2|s|}\big)},\qquad s\in\RR\setminus\{0\}.$$
Using these formulas we see that $\sup_{s\in\RR}\|P_2(1+is)\|<\infty$ and $\R P_2(is)\gtrsim(1+|s|)^{-1/2}$ for all $s\in\RR$. 
We may therefore apply Corollary~\ref{cor:decay}, taking $M_0$, $M_2$ and $\mu$ to be constant, and $\eta(s)=c(1+|s|)^{-1/2}$ for some $c>0$ and all $s\in\RR$, allowing us to set $M(s)=1+s^{1/2}$ for $s\ge0$. We conclude that the operator $A=\diag(L_1,L_2)$ with domain 
$$\begin{aligned}
D(A)=\big\{(u,v,w)^\top\in\ &H^1(-1,0)\times H^1(-1,0)\times H^2(0,1): \\&u(-1)=w(1)=0,\ v(0)=w(0),\ u(0)=w'(0)\big\},
\end{aligned}$$
which describes the wave-heat system~\eqref{eq:wave-heat} 
generates a strongly stable contaction semigroup $\T$ on $X$.
Corollary~\ref{cor:decay} furthermore shows that $\|T(t)z_0\|=o(t^{-2})$ as $t\to\infty$ for all $z_0\in D(A)$. By our choice of the norm on the space $X$ the energy of the (mild) solution of~\eqref{eq:wave-heat} with initial data $z_0\in X$ is given by $E(t)=\frac12\|T(t)z_0\|^2_X$ for all $t\ge0$, so we obtain the following result concerning energy decay in the wave-heat system~\eqref{eq:wave-heat}.

\begin{thm}\label{thm:wave-heat}
The energy of every  solution of the wave-heat system~\eqref{eq:wave-heat} satisfies $E(t)\to0$ as $t\to\infty$, and for classical solutions we have $E(t)=o(t^{-4})$ as $t\to\infty$.
\end{thm}

We thus recover the main result of~\cite{BatPau16}; see also~\cite{ZhaZua04}. As is noted in~\cite[Rem.~4.3(a)]{BatPau16}, this rate of energy decay for classical solutions is sharp. The result of Theorem~\ref{thm:wave-heat} remains true if we allow suitable non-constant coefficients in the wave part; see~\cite{CoxZua95} for an analysis of the wave part in this case, and~\cite{PauSOTA18} for results on the corresponding wave-heat system.  We may also replace the Dirichlet boundary condition in the heat part by a Neumann boundary condition. In this case  the origin becomes an isolated simple eigenvalue of $A$, but we may nevertheless  apply~Theorem~\ref{thm:res} in order to show that $\|(is-A)\inv\|\lesssim1+|s|^{1/2}$ for all $s\in\RR$ such that $|s|\ge1$, leading to the same decay rate as in Theorem~\ref{thm:wave-heat}, but this time for convergence towards a  solution with non-zero energy.

\subsection{Wave-heat networks}

Next we consider networks made up of one-dimensional wave equations and a single one-dimensional heat equation. In view of the more concrete application we have in mind, and in order to avoid unnecessary complication, we restrict our exposition to the important case of \emph{star-shaped} networks of wave equations. However, the reader will have no difficulty adapting our main result in this section, Theorem~\ref{thm:wave-heat-network}, to the setting of more general networks; see also Remark~\ref{rem:network} below.   To be precise, then, given $N\ge1$ and $\ell_0,\dots,\ell_{N}>0$ we consider the following system:
\begin{equation}\label{eq:wave-heat-network}
\left\{
\begin{aligned}
y^{k}_{tt}(x,t)&=y_{xx}^{k}(x,t), \hspace{-50pt}& x\in(0,\ell_k),\ t>0,\ 0\le k\le N,\\
w_t(x,t)&=w_{xx}(x,t),\hspace{-50pt}& x\in(0,1),\ t>0,\\
y^{k}(0,t)&=y^{0}(0,t),\quad y^{k}(\ell_k,t)=0,\hspace{-50pt}&t>0,\ 1\le k\le N,\\
\textstyle{\sum_{k=0}^{N}}y^{k}_x(0,t)&=0,\quad w(1,t)=0,\hspace{-50pt}&t>0,\\
y_t^{0}(\ell_0,t)&=w(0,t),\quad  y_x^{0}(\ell_0,t)=w_x(0,t), \hspace{-140pt}&t>0,
 \end{aligned}
\right.
\end{equation}
subject to suitable initial conditions on $y^{0},\dots,y^{N}$ and $w$. Note that the coupling condition for the wave equations at the centre of the star is of \emph{Kirchhoff} type. We wish to study the  asymptotic behaviour as $t\to\infty$ of the total energy $E(t)=E_W(t)+E_H(t)$ of our system, where
$$E_W(t)=\frac12\sum_{k=0}^{N}\int_0^{\ell_k}|y^{k}_x(x,t)|^2+|y^{k}_t(x,t)|^2\,\dd x,\qquad t\ge0,$$
and 
$$E_H(t)=\frac12\int_0^1|w(x,t)|^2\,\dd x,\qquad t\ge0,$$
denote the energies of the wave part and the heat, respectively. In order to formulate our system as an abstract Cauchy problem we consider the Hilbert spaces $H=\prod_{k=0}^{N}L^2(0,\ell_k)$ and 
$$V=\bigg\{(y_k)_{k=0}^{N}\in \prod_{k=0}^{N}H^1(0,\ell_k):y_k(0)=y_0(0),\ y_k(\ell_k)=0\mbox{ for }1\le k\le N\bigg\},$$
the latter endowed with the norm defined by $\|y\|_V^2=\sum_{k=0}^N\|y_k'\|^2_{L^2}$ for all $y=(y_k)_{k=0}^{N}\in V$.
Furthermore, let $X_1=V\times H$ and $X_2=L^2(0,1)$, and let $X= X_1\times X_2$. 
We may now reformulate our system as an abstract Cauchy problem of the form~\eqref{eq:ACP}  for the function 
$$z(t)=\big((y^{k}(\cdot,t))_{k=0}^N, (y_t^{k}(\cdot,t))_{k=0}^N,w\big)^\top,\qquad t\ge0,$$ on the Hilbert space $X$, where $A\colon D(A)\subseteq X\to X$ is  defined by $A(y,v,w)^\top=(v,y'',w'')^\top$ for  $(y,v,w)^\top$ in the domain
$$\begin{aligned}
D(A)=&\left\{\begin{pmatrix}(y_k)_{k=0}^{N}\\
(v_k)_{k=0}^{N}\\
w
\end{pmatrix}\in \bigg( V\cap\prod_{k=0}^{N}H^2(0,\ell_k)\bigg)\times V\times H^2(0,1):\right.\\
&\qquad\left. \sum_{k=0}^Ny_k'(0)=0,\ w(1)=0,\ v_0(\ell_0)=w(0),\ y_0'(\ell_0)=w'(0)\right\}.
\end{aligned}$$
Our aim is to relate the energy $E(t)$ of the coupled wave-heat system to the decay of the energy
$$E_0(t)=\frac12\sum_{k=0}^{N}\int_0^{\ell_k}|y^{k}_x(x,t)|^2+|y^{k}_t(x,t)|^2\,\dd x,\qquad t\ge0,$$
where $y^{0},\dots, y^{N}$ are assumed to satisfy the relevant equations in~\eqref{eq:wave-heat-network}, except with the coupling conditions imposed on $y^{0}$ at $x=\ell_0$ replaced by the non-reflective boundary condition $y_x^{0}(\ell_0,t)=-y_t^{0}(\ell_0,t)$ for all $t>0$. We shall refer to this system as the \emph{boundary-damped} network of wave equations, and we note that it corresponds to the abstract Cauchy problem 
$$\left\{\begin{aligned}
	  \dot{z}(t) &= A_0 z(t),  \qquad t\ge0,\\
	  	  z(0) &=z_0  
	\end{aligned}\right.
	$$
on $X_1$, where  
$$z(t)=\big((y^{k}(\cdot,t))_{k=0}^N, (y_t^{k}(\cdot,t))_{k=0}^N\big)^\top,\qquad t\ge0,$$
and $A_0\colon D(A_0)\subseteq X_1\to X_1$ is the operator defined by $A_0(y,v)^\top=(v,y'')^\top$ for  $(y,v)^\top$ in the domain
$$\begin{aligned}
D(A_0)=\bigg\{&\begin{pmatrix}(y_k)_{k=0}^{N}\\
(v_k)_{k=0}^{N}
\end{pmatrix} \in \bigg( V\cap\prod_{k=0}^{N}H^2(0,\ell_k)\bigg)\times V:\\
&\qquad\sum_{k=0}^Ny_k'(0)=0,\ y_0'(\ell_0)=-v_0(\ell_0)\bigg\}.
\end{aligned}$$

We now describe the rate of energy decay in the wave-heat network~\eqref{eq:wave-heat-network} in terms of the rate of energy decay in the boundary-damped network of wave equations. In general, for classical solutions $E(t)$ will decay slightly more slowly as $t\to\infty$ than  $E_0(t)$.

\begin{thm}\label{thm:wave-heat-network} Suppose that $r_0\colon \RR_+\to(0,\infty)$ is a continuous strictly decreasing function such that $r_0(t)\to0$ as $t\to\infty$ and $E_0(t) =O(r_0(t))$ as $t\to\infty$ for all classical solutions of the boundary-damped network of wave equations. Then $E(t)\to0$ as $t\to\infty$ for all solutions of the wave-heat network~\eqref{eq:wave-heat-network}, and for classical solutions $$E(t)=O(r(t)),\qquad t\to\infty,$$ where  $r$ is the inverse of $s\mapsto s^{-1/4}r_0\inv(s)$ for  $0<s\le r_0(0)$. Furthermore, if $E_0(t)=O(t^{-\alpha})$ as $t\to\infty$ for some $\alpha>0$ then $E(t)=o(t^{-4\alpha/(4+\alpha)})$ as $t\to\infty$ for all classical solutions of~\eqref{eq:wave-heat-network}.
\end{thm}

\begin{proof}
We apply  Corollary~\ref{cor:decay}. We begin by expressing  the operator $A$ in terms of two impedance passive boundary nodes $(G_1,L_1,K_1)$ and $(G_2,L_2,K_2)$ on $(\C,X_1,\C)$ and $(\C,X_2,\C)$ corresponding to the wave network and to the heat equation, respectively, and with $J=1$ and $Q=1$.
Specifically, we define  $L_j\colon D(L_j)\subseteq X_j\to X_j$, for $j=1,2$, by $L_1(y,v)^\top=(v,y'')^\top$ with 
$$D(L_1)=\bigg\{\begin{pmatrix}(y_k)_{k=0}^{N}\\
(v_k)_{k=0}^{N}
\end{pmatrix}
 \in \bigg( V\cap\prod_{k=0}^{N}H^2(0,\ell_k)\bigg)\times V:\sum_{k=0}^Ny_k'(0)=0\bigg\}$$
 and $L_2w=w''$ with $D(L_2)=\{w\in H^2(0,1):w(1)=0\}$. Furthermore, let $G_j, K_j\colon D(L_j)\subseteq X_j\to \C$, for $j=1,2$, be defined by 
$$G_1\begin{pmatrix}y\\v\end{pmatrix}=v_0(\ell_0)\qquad\mbox{and}\qquad K_1\begin{pmatrix}y\\v\end{pmatrix}=y_0'(\ell_0)$$ 
for $(y,v)^\top=((y_k)_{k=0}^{N},(v_k)_{k=0}^{N})^\top\in D(L_1)$, and $G_2w=-w'(0)$ and $K_2w=w(0)$ for $w\in D(L_2)$, respectively. Then the operator $A$ governing the wave-heat network~\eqref{eq:wave-heat-network} may be written as $A=\diag(L_1,L_2)$ with $D(A)=\{(y,v,w)^\top\in D(L_1)\times D(L_2): G_1(y,v)^\top=K_2w, G_2w=-K_1(y,v)^\top\}$. As was shown in Section~\ref{sec:wave_heat}, $(G_2,L_2,K_2)$ is an impedance passive boundary node on $(\C,X_2,\C).$ We now show that $(G_1,L_1,K_1)$ is an impedance passive boundary node on $(\C,X_1,\C)$. For $x=(y,v)^\top\in D(L_1)$, where $y=(y_k)_{k=0}^{N}$ and $v=(v_k)_{k=0}^{N}$, integration by parts gives
\begin{equation}\label{eq:network_ip}
\begin{aligned}
\R\langle L_1 x,x\rangle& =\R \sum_{k=0}^N \big(\langle v_k',y_k'\rangle_{L^2}+\langle y_k'',v_k\rangle_{L^2}\big)\\
&=\R v_0(\ell_0)\overline{y_0'(\ell_0)}-\R v_0(0)\sum_{k=0}^N\overline{y_k'(0)}\\
&=\R \langle G_1x,K_1 x\rangle_\C,
\end{aligned}
\end{equation}
so the restriction $A_1$ of $L_1$ to $\Ker G_1$ is dissipative. By the Rellich--Kondrachov theorem $A_1$ has compact resolvent and in particular is maximally dissipative, so by the Lumer--Phillips theorem it generates a contraction semigroup on $X_1$. Since the operator $G_1$ is non-zero it maps onto its codomain $\C$ and in particular admits a right-inverse  $G_1^r\in\B(\C,D(L_1))$.  Thus $(G_1,L_1,K_1)$ is a boundary node, and by~\eqref{eq:network_ip} it is impedance passive.

By our choice of norm on $X_1$ and our assumption on the energy $E_0$ of the network of boundary-damped wave equations we have $\|T_0(t)z_0\|=O(r_0(t)^{1/2})$ as $t\to\infty$ for all $z_0\in D(A_0)$, where $(T_0(t))_{t\ge0}$ is the contraction semigroup on $X_1$ generated by $A_0$. It follows from the uniform boundedness principle that if $\gl\in\rho(A_0)$ then $\|T_0(t)(\gl-A_0)\inv\|=O(r_0(t)^{1/2})$ as $t\to\infty$, and hence $i\RR\subseteq\rho(A_0)$ and $\|(is-A_0)\inv\|=O(r_0^{-1}(c|s|^{-2}))$ as $|s|\to\infty$ for some $c>0$, by Remark~\ref{rem:asymp}. As was shown in Section~\ref{sec:wave_heat}, $i\RR\subseteq\rho(A_2)$ and $\sup_{s\in\RR}\|(is-A_2)\inv\|<\infty$, and the transfer function $P_2$ of the boundary node $(G_2,L_2,K_2)$ satisfies $\sup_{s\in\RR}\|P_2(1+is)\|<\infty$ and $\R P_2(is)\gtrsim(1+|s|)^{-1/2}$ for all $s\in\RR$. We now apply Corollary~\ref{cor:decay}. We may take the function $M_0\colon\RR\to(0,\infty)$ to be defined, for suitable constants $C,s_0>0$, by $M_0(s)=Cr_0^{-1}(c|s|^{-2})$ when $|s|> s_0$ and by $M_0(s)=Cr_0^{-1}(c|s_0|^{-2})$ for $|s|\le s_0$. Furthermore, we may take $M_2$ and $\mu$ to be constant, and $\eta(s)=\delta(1+|s|)^{-1/2}$ for some $\delta>0$ and all $s\in\RR$. We may therefore take  $M\colon\RR_+\to(0,\infty)$ to be defined by $M(s)=(1+s)^{1/2}M_0(s)$, $s\ge0$, which is a continuous function that has positive increase. We conclude that $A$ generates a strongly stable contraction semigroup $\T$ on $X$. We note that, by the choice of norm on $X$, the energy of the solution with initial data $z_0\in X$ is given by $E(t)=\frac12\|T(t)z_0\|^2$ for all $t\ge0$, and hence $E(t)\to0$ as $t\to\infty$ for all solutions of~\eqref{eq:wave-heat-network}. By~\cite[Prop.~2.2]{RoSeSt19} we have $M\inv(at)=O(M\inv(t))$   as $t\to\infty$ for all $a>0$,  so Corollary~\ref{cor:decay}(c) yields $E(t)=O(r(t))$ as $t\to\infty$ for all classical solutions, where $r$ is the inverse of $s\mapsto s^{-1/4}r_0\inv(s)$ for  $0<s\le r_0(0)$. Finally, if $r_0(t)=O(t^{-\alpha})$ as $t\to\infty$ for some $\alpha>0$ then we may take $M(s)=1+s^{(4+\alpha)/2\alpha}$  for all $s\ge0$ in Corollary~\ref{cor:decay} to obtain that $E(t)=o(t^{-4\alpha/(4+\alpha)})$ as $t\to\infty$ for all classical solutions.
\end{proof}

As an application of Theorem~\ref{thm:wave-heat-network} we consider a special class of star-shaped networks in which the ratios $\ell_k/\ell_j$ for $1\le k<j\le N$ are irrational numbers of \emph{constant type}, which is to say that the coefficients appearing in the continued fraction expansion form a bounded sequence. Such irrationals may be thought of as being poorly approximable by rational numbers, and they are known to form an uncountable subset of $\RR$ of measure zero.

\begin{cor}
Consider the wave-heat network~\eqref{eq:wave-heat-network} with $N\ge2$, and suppose that $\ell_k/\ell_j$ is an irrational number of constant type for $1\le k<j\le N$. Then $E(t)\to0$ as $t\to\infty$ for all solutions, and for classical solutions 
\begin{equation}\label{eq:network_energy}
E(t)=o\big(t^{-\frac{4}{4N-3}}\big),\qquad t\to\infty.
\end{equation}
\end{cor}

\begin{proof}
Let $A_0\colon D(A_0)\subseteq X_1\to X_1$ be the operator corresponding to the boundary-damped network of wave equations. Since $A_0$ generates a contraction semigroup on $X_1$ it follows from~\cite[Sect.~4.3]{Lun18book} that for all $\alpha\in(0,1)$ the complex interpolation space $[D(A_0),X_1]_{1-\alpha}$ coincides with the domain of the fractional power $(-A_0)^\alpha$. By~\cite[Prop.~4.4]{ValZua09} we have $\|T_0(t)z_0\|=O(t^{-\alpha/(2N-2)})$ as $t\to\infty$ for all $z_0\in D((-A_0)^\alpha)$, provided that $\alpha\in(0,\frac12)$. For the choice $\alpha=1/3$, an application of the uniform boundedness principle gives $\|T_0(t)(-A_0)^{-1/3}\|=O(t^{-1/(6N-6)})$ as $t\to\infty$, and hence by the semigroup property $E_0(t)=O(t^{-1/(N-1)})$ as $t\to\infty$ for all classical solutions of the boundary-damped network of wave equations. It follows from Theorem~\ref{thm:wave-heat-network} that $E(t)\to0$ as $t\to\infty$ for all mild solutions of the wave-heat network~\eqref{eq:wave-heat-network}, and that $E(t)=o(t^{-{4}/{(4N-3)}})$ as $t\to\infty$ for all classical solutions.
\end{proof}

\begin{rem}\label{rem:network}
For simplicity we have considered star-shaped networks of wave equations only, but the same analysis can be carried out for other network structures; see for instance~\cite{AmmJel04, AmJeKh05,NicVal07, ValZua09}.
\end{rem}

\subsection{A wave equation with an acoustic boundary condition}

As our final PDE model we study a multi-dimensional wave equation with an \emph{acoustic boundary condition}, following \cite{AbbNic15b, Bea76, RivQin03}. Given $n\ge 2$ let $\Omega$ be a bounded domain in $\RR^n$ with $C^2$ boundary $\partial\Omega=\Gamma_0\cup\Gamma_1$, where $\Gamma_0,\Gamma_1$ are assumed to be closed, non-empty and  disjoint. We consider the problem
\begin{equation}\label{eq:wave_ac}
\left\{
\begin{aligned}
y_{tt}(x,t)&=\Delta y(x,t), &x\in \Omega,\ t>0,\\
y(x,t)&=0,&x\in\Gamma_0,\ t>0,\\
\partial_\nu y(x,t)&=p_t(x,t),&x\in\Gamma_1,\ t>0,\\
m(x)p_{tt}(x,t)&=-d(x)p_t(x,t)-k(x)p(x,t)-y_t(x,t), &x\in\Gamma_1,\ t>0,
\end{aligned}
\right.
\end{equation}
to be solved subject to suitable initial conditions for $y$, $y_t$, $p$ and $p_t$. Here $\partial_\nu$ denotes the outward normal derivative and the functions $m,d,k\in L^\infty(\Gamma_1)$ are assumed to be such that $m(x)\ge m_0$, $d(x)\ge d_0$ and $k(x)\ge k_0$ for some constants $m_0,d_0,k_0>0$ and almost all $x\in\Gamma_1$. Let $H_{\Gamma_0}^1(\Omega)=\{y\in H^1(\Omega):y|_{\Gamma_0}=0\}$ and let $X=H_{\Gamma_0}^1(\Omega)\times L^2(\Omega)\times L^2(\Gamma_1)^2$ with the inner product
$$\langle x_1,x_2\rangle=\langle\nabla y_1,\nabla y_2 \rangle_{L^2(\Omega)}+\langle v_1,v_2 \rangle_{L^2(\Omega)}+\langle kp_1,p_2 \rangle_{L^2(\Gamma_1)}+\langle mq_1,q_2 \rangle_{L^2(\Gamma_1)}$$
for $x_j=(y_j,v_j,p_j,q_j)^\top\in X$, $j=1,2$. Then we may formulate our problem~\eqref{eq:wave_ac} as an abstract Cauchy problem on the Hilbert space $X$, with $A\colon D(A)\subseteq X\to X$ defined by 
\begin{equation}\label{eq:A_def}
A \begin{pmatrix}
y\\v\\ p\\ q
\end{pmatrix}
=\begin{pmatrix}v\\\Delta y\\q\\
-m\inv (dq+kp+v|_{\Gamma_1})
\end{pmatrix}
\end{equation}
for $(y,v,p,q)^\top$ in the domain $D(A)=\{(y,v,p,q)^\top\in  H_{\Gamma_0}^1(\Omega)^2\times L^2(\Gamma_1)^2: \Delta y\in L^2(\Omega),\ \partial_\nu y|_{\Gamma_1}=q \}$ of $A$. Here we interpret $\partial_\nu y|_{\Gamma_1}=q$ as an equality in $H^{-1/2}(\Gamma_1)$, the (conjugate) dual space of $H^{1/2}(\Gamma_1)$ with respect to the pivot space $L^2(\Gamma_1)$. We wish to study the asymptotic behaviour of the solutions of~\eqref{eq:wave_ac} by  relating the resolvent growth of $A$ along the imaginary axis to that of the operator $A_0\colon D(A_0)\subseteq X_1\to X_1$, where $X_1=H_{\Gamma_0}^1(\Omega)\times L^2(\Omega)$ is endowed with the inner product  
$$\langle x_1,x_2\rangle=\langle\nabla y_1,\nabla y_2 \rangle_{L^2(\Omega)}+\langle v_1,v_2 \rangle_{L^2(\Omega)}$$
for $x_j=(y_j,v_j)^\top\in X_1$, $j=1,2$, and
$$A_0 \begin{pmatrix}
y\\v
\end{pmatrix}
=\begin{pmatrix}v\\\Delta y 
\end{pmatrix}$$
for $(y,v)^\top$ in the domain 
$$
D(A_0)=\bigg\{\begin{pmatrix}y\\ v 
\end{pmatrix}\in  H_{\Gamma_0}^1(\Omega)^2:   \Delta y\in L^2(\Omega),\ \partial_\nu y|_{\Gamma_1}+v|_{\Gamma_1}=0\bigg\},$$
where $\partial_\nu y|_{\Gamma_1}+v|_{\Gamma_1}=0$ is interpreted as an equality in $H^{-1/2}(\Gamma_1)$. Note that $A_0$ governs the evolution of the \emph{boundary-damped} wave equation in the Hilbert space $X_1$, that is to say the wave equation on $\Omega$ subject to a Dirichlet boundary condition on $\Gamma_0$ and the dissipative boundary condition $\partial_\nu y=-y_t$ on $\Gamma_1$. We define the energy of a solution of~\eqref{eq:wave_ac} by
$$E(t)=\frac12\int_\Omega|\nabla y(x,t)|^2+|y_t(x,t)|^2\,\dd x+\frac12\int_{\Gamma_1}k(s)|p(s,t)|^2+m(s)|p_t(s,t)|^2\,\dd s$$
for all $t\ge0$. Our next result describes the asymptotic behaviour of this energy given an estimate for the resolvent of  the generator $A_0$ of the boundary-damped wave system.

\begin{thm}\label{thm:ac_wave}
 Suppose that $i\RR\subseteq\rho(A_0)$ and let $M_0\colon\RR_+\to(0,\infty)$ be a continuous non-decreasing function such that 
$$\|(is-A_0)\inv\|\le M_0(|s|),\qquad s\in\RR.$$ 
Then $E(t)\to0$ as $t\to\infty$ for all solutions of~\eqref{eq:wave_ac}, and for classical solutions
$$E(t)=O(r(t)),\qquad t\to\infty,$$ 
where $r$ is the inverse of $s\mapsto s\inv M_0(s^{-1/2})$ for $0<s\le1$. Furthermore, if there exists $\alpha\ge0$ such that $\|(is-A_0)\inv\|\lesssim1+|s|^\alpha$ for all $s\in\RR$, then $E(t)=o(t^{-2/(2+\alpha)})$ as $t\to\infty$ for all classical solutions of~\eqref{eq:wave_ac}.
\end{thm}

\begin{proof}
Define the operators $L_1\colon D(L_1)\subseteq X_1\to X_1$ and $G_1, K_1\colon D(L_1)\subseteq X_1\to H^{-1/2}(\Gamma_1)$ by 
$$L_1 \begin{pmatrix}
y\\v
\end{pmatrix}
=\begin{pmatrix}v\\\Delta y 
\end{pmatrix},\qquad G_1 \begin{pmatrix}
y\\v
\end{pmatrix}
=\partial_\nu y|_{\Gamma_1},\qquad K_1 \begin{pmatrix}
y\\v
\end{pmatrix}
=\Phi(v|_{\Gamma_1})
$$
for $(y,v)^\top$ in the domain $D(L_1)=\{(y,v)^\top\in  H_{\Gamma_0}^1(\Omega)\times H_{\Gamma_0}^1(\Omega): \Delta y\in L^2(\Omega)\}$. Here $\Phi\colon H^{1/2}(\Gamma_1)\to H^{-1/2}(\Gamma_1)$ denotes the unitary map provided by the Riesz--Fréchet theorem.
 Then by \cite[Prop.~6.2]{MalSta07} and Remark~\ref{rem:Riesz} the triple $(G_1,L_1,K_1)$ is an impedance passive boundary node on $(U_1,X_1,U_1)$, where $U_1=H^{-1/2}(\Gamma_1)$. 

Let $U_2=L^2(\Gamma_1)$ and
define $Q=I\in \B(U_2)$.
Write
 $J_1\colon H^{1/2}(\Gamma_1)\to U_2$ and $J_2\colon U_2\to U_1$ for the natural embeddings, and define $J\in\B(U_1,U_2)$ by $J=J_1\Phi^*$. A simple calculation using the fact that $U_1$ is the (conjugate) dual space of $H^{1/2}(\Gamma_1)$ with respect to the pivot space $U_2$ shows that $J_2=\Phi J_1^*$, and hence $J^*J=J_2J_1\Phi^*.$
Thus the operator $A_0$ that governs the evolution of the boundary-damped wave equation is precisely the restriction of $L_1$ to $\Ker(G_1+J^*JK_1)$, and hence the function $M_0$ plays the same role as the function $M_0$ appearing in Theorem~\ref{thm:coupled}.
Let  $X_2=U_2\times U_2=L^2(\Gamma_1)^2$, endowed with the inner product 
$$\langle x_1,x_2\rangle_{X_2}=\langle kp_1,p_2 \rangle_{L^2(\Gamma_1)}+\langle mq_1,q_2 \rangle_{L^2(\Gamma_1)}$$
for $x_j=(p_j,q_j)^\top\in X_2$, $j=1,2$, and define
 the operators $A_2\in\B(X_2)$, $B_2\in\B(U_2,X_2)$, $C_2\in\B(X_2,U_2)$ and $D_2\in\B(U_2)$ by 
$$A_2 \begin{pmatrix}
p\\q
\end{pmatrix}
=\begin{pmatrix}q\\
- m\inv (dq+kp)
\end{pmatrix},\qquad B_2 u
=\begin{pmatrix}0\\
m\inv  u
\end{pmatrix},\qquad C_2 \begin{pmatrix}
p\\q
\end{pmatrix}
=q
$$
and $D_2=0$. It is straightforward to verify that the operator $A$ defined in \eqref{eq:A_def} satisfies the description of  the operator $A$ considered in Theorem~\ref{thm:coupled}.
Note that $\gl -A_2$ is invertible, and in particular surjective, for all $\gl>\|A_2\|$.  Hence for $x=(p,q)^\top\in X_2$ and $u\in U_2$ we have
$$\begin{aligned}
\R\langle A_2 x+B_2u,x\rangle_{X_2}&=-\|d^{1/2}q\|_{L^2(\Gamma_1)}^2+\R\langle  u,q\rangle_{L^2(\Gamma_1)}\\&=-\|d^{1/2}q\|_{L^2(\Gamma_1)}^2+\R\langle q,u\rangle_{U_2}\\&\le \R\langle C_2x+D_2u,u\rangle_{U_2},
\end{aligned}$$
so~\eqref{eq:imped_pass} is satisfied.  It is straightforward to verify that $i\RR\subseteq\rho(A_2)$ and that $\sup_{s\in \RR}\|(is-A_2)\inv\|<\infty$. Let $P_2(\gl)=C_2(\gl-A_2)\inv B_2$ for $\gl\in\rho(A_2)$. Then 
$P_2(is)=s(sd+i(s^2m-k))\inv I$
 and hence
$$
\R P_2(is)= \frac{s^2d}{s^2d+(ms^2-k)^2}I\ge\frac{s^2d_0}{s^2\|d\|_\infty+2s^4\|m\|_\infty^2+2\|k\|_\infty^2}I
$$
for all $s\in \RR$. In particular, $P_2(0)=0$. Let $s_0>0$. Then there exists $c>0$ such that $\R P_2(is)\ge c (1+|s|^2)^{-1}I$ for $|s|\ge s_0$. 
We now apply Theorem~\ref{thm:coupled} with $E=\{s\in\RR:|s|\ge s_0\}$, $M_0$ as above, $M_2$ constant and $\eta(s)=c (1+|s|^2)^{-1}$ for $s\in E$. We conclude that $A$ generates a contraction semigroup $\T$ on $X$, $iE\subseteq\rho(A)$ and $\|(is-A)\inv\|\lesssim M_0(|s|)(1+|s|^2)$ for $|s|\ge s_0$. Since $P_2(0)=0$, Remark~\ref{rem:dyn} shows that $0\in\rho(A)$. Since the resolvent set is open and $s_0>0$ was arbitrary, we deduce that $i\RR\subseteq\rho(A)$. Furthermore, 
 $\|(is-A)\inv\|\lesssim M_0(|s|)(1+|s|^2)$ for all $s\in\RR$. Finally, if $\|(is-A_0)\inv\|\lesssim1+|s|^\alpha$  then $\|(is-A)\inv\|\lesssim 1+|s|^{2+\alpha}$ for all $s\in\RR$. Since the function $s\mapsto M_0(s)(1+s^2)$ has positive increase, the result now follows from~Theorem~\ref{thm:sg}.
\end{proof}

Finally, we apply Theorem~\ref{thm:ac_wave} in order to study the rate of energy decay of waves subject to an acoustic boundary condition as in~\eqref{eq:wave_ac}. We restrict our attention to an illustrative special case.

\begin{ex}
Let $\Omega=\{x\in\RR^2:1<|x|<2\}$ and let $\Gamma_0$, $\Gamma_1$ be the inner and the outer part of the boundary of the annulus $\Omega$, in either order. We consider both cases in turn.
\begin{enumerate}[(a)]
\item If $\Gamma_0=\{x\in\RR^2:|x|=1\}$ is the inner boundary and the acoustic boundary condition is applied along the outer boundary $\Gamma_1=\{x\in\RR^2:|x|=2\}$ then the \emph{geometric control condition} is satisfied in the corresponding boundary-damped wave equation on $\Omega$ (with damping on $\Gamma_1$), and hence the semigroup generated by $A_0$ is exponentially stable; see~\cite{BaLeRa92}. It follows that $i\RR\subseteq\rho(A_0)$ and $\sup_{s\in\RR}\|(is-A_0)\inv\|<\infty$. Hence Theorem~\ref{thm:ac_wave} gives $E(t)\to0$ as $t\to\infty$ for all solutions of~\eqref{eq:wave_ac}, and $E(t)=o(t^{-1})$ as $t\to\infty$ for all classical solutions.
\item If $\Gamma_0=\{x\in\RR^2:|x|=2\}$ is the outer boundary and the acoustic boundary condition is applied along the inner boundary $\Gamma_1=\{x\in\RR^2:|x|=1\}$ of $\Omega$, then $i\RR\subseteq\rho(A_0)$ as before but the presence of \emph{whispering gallery modes} means that one cannot expect a better resolvent bound than $\|(is-A_0)\inv\|\lesssim e^{c|s|}$ for some $c>0$ and all $s\in\RR$; see for instance~\cite{LebRob97}. It follows from  Theorem~\ref{thm:ac_wave} that $E(t)\to0$ as $t\to\infty$ for for all solutions of~\eqref{eq:wave_ac}, and  that $E(t)=O(\log(t)\inv)$ as $t\to\infty$ for all classical solutions.
\end{enumerate}
\end{ex}

\end{document}